\documentclass[a4paper,12pt,intlimits,oneside,reqno]{amsart}
\usepackage{enumerate}
\usepackage{amsfonts,amsmath}

\usepackage{latexsym,amssymb}
\newtheorem{theorem}{Theorem}[section]

\newtheorem{lemma}[theorem]{Lemma}
\theoremstyle{corollary}
\newtheorem{corollary}[theorem]{Corollary}
\theoremstyle{proposition}
\newtheorem{proposition}[theorem]{Proposition}

\theoremstyle{definition}
\newtheorem{definition}[theorem]{Definition}

\theoremstyle{remark}

\numberwithin{equation}{section}

\newcommand{\comment}[1]{}

\begin{document}

\title [multilinear  Hausdorff operators and commutators]{Weighted Lebesgue and central Morrey 
\\
estimates for p-adic multilinear  Hausdorff operators and its commutators}

\thanks{This paper is funded by Vietnam National Foundation for Science and Technology Development (NAFOSTED)}

\author{Nguyen Minh Chuong}

\address{Institute of mathematics, Vietnamese  Academy of Science and Technology,  Hanoi, Vietnam.}
\email{nmchuong@math.ac.vn}

\author{Dao Van Duong}
\address{School of Mathematics, Mientrung University of Civil Engineering, Phu Yen, Vietnam}
\email{daovanduong@muce.edu.vn}

\author{Kieu Huu Dung}
\address{School of Mathematics, University of Transport and Communications, Ha Noi, Vietnam}
\email{khdung@utc2.edu.vn}
\keywords{Multilinear  Hausdorff operator, commutator, central BMO space, Morrey space, $A_p$ weight, maximal operator, $p$-adic analysis.}
\subjclass[2010]{42B25, 42B99, 26D15}
\begin{abstract}
In this paper, we establish the sharp boundedness of $p$-adic multilinear Hausdorff operators on the product of Lebesgue and central Morrey spaces associated with both power weights and  Muckenhoupt weights. Moreover, the boundedness for the commutators of $p$-adic multilinear Hausdorff operators on the such spaces with symbols in central BMO space is also obtained.
\end{abstract}

\maketitle
\section{Introduction}
The $p$-adic analysis in the past decades has received a lot of attention due to its important applications in mathematical physics as well as its necessity in sciences and technologies (see e.g. \cite{Avetisov1, Avetisov2, Chuong1, Dragovich, Khrennikov1, K2001, Kozyrev, Varadarajan, Vladimirov, Vladimirov1, Vladimirov2} and references therein). It is well known that the theory of functions from $\mathbb Q_p$ into $\mathbb C$ play an important role in $p$-adic quantum mechanics, the theory of $p$-adic probability in which real-valued random variables have to be considered to solve covariance problems. In recent years, there is an increasing interest in the study of harmonic analysis and wavelet analysis over the $p$-adic fields (see e.g. \cite{Albeverio, Chuong1, Chuong5, Haran1, Haran2, Kozyrev}).
\vskip 5pt
 It is crucial that the Hausdorff operator is one of the important operators in harmonic analysis. It is closely related to the summability of the classical Fourier series  (see, for instance, \cite{Dyachenko}, \cite{Hausdorff}, \cite{Hurwitz},  and the references therein).  Let $\Phi$ be a locally integrable function on $\mathbb R^n$. The matrix Hausdorff operator $H_{\Phi,A}$ associated to the kernel function $\Phi$ is then defined in terms of the integral form as follows
\begin{equation}\label{Hausdorff1}
H_{\Phi, A}(f)(x)=\int\limits_{\mathbb R^n}{\frac{\Phi(y)}{|y|^n}f(A(y) x)dy},\,x\in\mathbb R^n,
\end{equation}
where $A(y)$ is an $n\times n$ invertible matrix for almost everywhere $y$ in the support of $\Phi$. It is worth pointing out that if the kernel function $\Phi$ is chosen appropriately, then the Hausdorff operator reduces to many classcial operators in analysis such as the Hardy operator, the Ces\`{a}ro operator, the Riemann-Liouville fractional integral operator and the Hardy-Littlewood average operator.
 \vskip 5pt
In 2010, Volosivets \cite{Volosivets1} introduced the matrix Hausdorff operator on the $p$-adic numbers field as follows
\begin{equation}\label{HausdorfVolosivets}
{\mathcal H}_{\varphi,A}(f)(x)=\int_{\mathbb Q_p^n} \varphi(t)f(A(t)x)dt,\;\;x\in\mathbb Q^n_p,
\end{equation}
where $\varphi(t)$ is a locally integrable function on $\mathbb Q_p^n$ and $A(t)$ is an $n\times n$ invertible matrix for almost everywhere $t$ in the support of $\varphi$.  It is easy to see that if $\varphi(t)=\psi(t_1)\chi_{{\mathbb Z^*_p}^n}(t)$ and $A(t)= t_1.I_n$ ($I_n$ is an identity matrix), for $t=(t_1,t_2,...,t_n)$, where $\psi:\mathbb Q_p\to \mathbb C$ is a measurable function,  ${\mathcal H}_{\varphi,A}$ then reduces to the $p$-adic weighted Hardy-Littlewood average operator due to Rim and Lee \cite{Rim}. In recent years, the theory of  the Hardy operators, the Hausdorff operators over the $p$-adic numbers field has been significantly developed into different contexts, and they are actually useful for $p$-adic analysis (see e.g. \cite{Chuong3}, \cite{Chuongduong}, \cite{Hung}, \cite{Volosivets3}). It is known that the authors in \cite{CDD2017} also introduced and studied a general class of multilinear Hausdorff operators on the real field defined by
\begin{equation}\label{mulHausdorff}
{\mathcal{H}_{\Phi ,\vec{A} }}(\vec{f})(x) = \int\limits_{{\mathbb R^n}} {\frac{{\Phi (y)}}{{{{\left| y \right|}^n}}}} \prod\limits_{i = 1}^m {{f_i}} ({A_i}(y)x)dy,\,x\in\mathbb R^n,
\end{equation}
for $\vec{f}=\left(f_1, ..., f_m\right)$ and $\vec{A}=\left(A_1, ..., A_m\right)$.
\vskip 5pt
Motivated by above results, in this paper we shall introduce and study a class of $p$-adic multilinear (matrix) Hausdorff operators defined as follows.
\begin{definition}
Let $\Phi: \mathbb Q_p^n\rightarrow [0, \infty) $. Let $f_1, f_2, ..., f_m$ be measurable complex-valued functions on $\mathbb Q^n_p$. The $p$-adic multilinear Hausdorff operator is defined by
\begin{equation}\label{Hpadic}
{\mathcal H}^p_{\Phi,\vec A}(\vec f)(x)=\int_{\mathbb Q^n_p} \dfrac{\Phi(y)}{|y|^n_p}\prod\limits_{i=1}^{m} f_i(A_i(y)x)dy,\;\;x\in\mathbb Q^n_p,
\end{equation}
for $\vec f=\big(f_1,..., f_m\big)$.
\end{definition}
\vskip 5pt
Let $b$ be a measurable function. We denote by $\mathcal{M}_b$ the multiplication operator defined  by $\mathcal{M}_bf (x)=b(x) f (x)$ for any measurable function $f$. If $\mathcal{H}$ is a linear operator on some measurable function space, the commutator of Coifman-Rochberg-Weiss type formed by $\mathcal{M}_b$  and $\mathcal{H}$ is defined by $[\mathcal{M}_b, \mathcal{H}]f (x)=(\mathcal{M}_b\mathcal{H}-\mathcal{H}\mathcal{M}_b) f (x)$. Analogously, let us give the definition for the commutators of Coifman-Rochberg-Weiss type of $p$-adic multilinear Hausdorff operator.
\begin{definition}
Let $\Phi, \vec A$ be as above. The Coifman-Rochberg-Weiss type commutator of $p$-adic multilinear Hausdorff operator is defined by

\begin{equation}\label{commuatatorHpadic}
{\mathcal H}^p_{\Phi,\vec A,\vec b}(\vec f)(x)=\int_{\mathbb Q^n_p} \dfrac{\Phi(y)}{|y|_p^n}\prod\limits_{i=1}^{m}\Big(b_i(x)- b_i(A_i(y)x)\Big)\prod\limits_{i=1}^{m} f_i(A_i(y)x)dy,
\end{equation}
where $x\in\mathbb Q^n_p $, $\vec b=\big(b_1,..., b_m\big)$ and $b_i$ are locally integrable functions on $\mathbb Q_p^n$ for all $i=1,...,m$.
\end{definition}
\vskip 5pt
The main purpose of this paper is to study the  $p$-adic multilinear Hausdorff operators and its commutators on the $p$-adic numbers field. More precisely,  we obtain the necessary and sufficient conditions for  the boundedness of ${\mathcal {H}}^p_{\Phi,\vec A}$ and ${\mathcal {H}}^p_{\Phi,\vec A,\vec b}$ on the  product of  Lebesgue and central Morrey spaces with weights on $p$-adic field. In each case, we estimate the corresponding operator norms. Moreover, the boundedness of ${\mathcal {H}}^p_{\Phi,\vec A,\vec b}$ on  the such spaces with symbols in central BMO space is also established. It should be pointed out that all our results  are new even in the case of $p$-adic linear Hausdorff operators.
\vskip 5pt
Our paper is organized as follows. In Section \ref{section2}, we present some notations and preliminaries about $p$-adic analysis as well as give some definitions of the Lebesgue and central Morrey spaces associated with power weights and  Muckenhoupt weights. Our main theorems are given and proved in Section \ref{section3} and Section \ref{section4}.

\section{Some notations and definitions}\label{section2}
For a prime number $p$, let $\mathbb Q_p$ be the field of $p$-adic numbers. This field is the completion of the field of rational numbers $\mathbb Q$ with respect to the non-Archimedean $p$-adic norm $|\cdot|_p$. This norm is defined as follows: if $x=0$, $|0|_p=0$; if $x\not=0$ is an arbitrary rational number with the unique representation $x=p^\alpha\frac{m}{n}$, where $m, n$ are not divisible by $p$, $\alpha=\alpha(x)\in\Bbb Z$, then $|x|_p=p^{-\alpha}$. This norm satisfies the following properties:

\rm{(i)} $|x|_p\geq 0,\;\; \forall x\in\mathbb Q_p$, and $|x|_p=0\Leftrightarrow x=0;$

\rm{(ii)} $|xy|_p=|x|_p|y|_p,\;\; \forall x, y\in\Bbb Q_p;$

\rm{(iii)} $|x+y|_p\leq \max(|x|_p,|y|_p),\; \forall x, y\in\mathbb Q_p$, and
when $|x|_p\not=|y|_p$, we have $|x+y|_p=\max(|x|_p,|y|_p).$

It is also well-known that any non-zero $p$-adic number $x\in\mathbb Q_p$ can be uniquely represented in the canonical series
\begin {equation}\label{eq1.1}
x=p^\alpha (x_0+x_1p+x_2p^2+\cdot\cdot\cdot),
\end{equation}
where $\alpha=\alpha(x)\in\Bbb Z$,\; $x_k=0,1,...,p-1,\; x_0\not= 0,\; k=0,1,...$.
This series converges in the $p$-adic norm since $|x_kp^k|_p\leq p^{-k}$.

The space $\mathbb Q^n_p=\mathbb Q_p\times\cdot\cdot\cdot\times\mathbb Q_p$ consists of all points $x=(x_1,...,x_n)$, where $x_i\in\mathbb Q_p,\; i=1,...,n,\;\; n\geq1$. The $p$-adic norm of $\mathbb Q^n_p$ is defined by
\begin {equation}\label{eq1.2}
|x|_p=\max_{1\leq j\leq n}|x_j|_p.
\end{equation}
Let $A$ be an $n\times n$ matrix with entries $a_{ij}\in\mathbb Q_p$.
For $x=(x_1, ..., x_n)\in\mathbb Q^n_p$, we denote
$$
Ax=\Big(\sum\limits_{j=1}^n a_{1j}x_j, ..., \sum\limits_{j=1}^n a_{nj}x_j \Big).
$$
By Lemma 2 in paper \cite{Volosivets3}, the norm of $A$, regarded as an operator from $\mathbb Q^n_p$ to $\mathbb Q^n_p$, is 
$$
\|A\|_p:=\mathop{\rm max}\limits_{1\leq i\leq n} \mathop{\rm max}\limits_{1\leq j\leq n}|a_{ij}|_p.
$$
For simplicity of notation, we write $k_{A}={\rm log}_p \|A\|_p$. It is clear to see that $k_A\in\mathbb Z$.  It is easy to show that $|Ax|_p\leq \|A\|_p.|x|_p$ for any $x\in\mathbb Q^n_p$. In addition, if $A$ is invertible, by estimating as Lemma 3.1 in paper \cite{RFW2017}, then we get
\begin{align}\label{|det|}
\|A\|_p^{-n}\leq |{\rm det}{(A^{-1})}|_p\leq \|A^{-1}\|_p^n.
\end{align}

Let 
\[B_\alpha(a)=\left\{x\in\mathbb Q^n_p : |x-a|_p\leq p^\alpha\right\}
\]
be a ball of radius $p^\alpha$ with center at $a\in\mathbb Q^n_p$. Similarly, denote by 
\[S_\alpha(a)=\left\{x\in\mathbb Q^n_p : |x-a|_p=p^\alpha\right\}\]
the sphere with center at $a\in\mathbb Q^n_p$ and radius $p^\alpha$. If $B_\alpha=B_\alpha(0), S_\alpha=S_\alpha(0)$, then for any $x_0\in\mathbb Q^n_p$ we have $x_0+B_\alpha=B_\alpha(x_0)$ and $x_0+S_\alpha=S_\alpha(x_0)$. 


Since $\mathbb Q^n_p$ is a locally compact commutative group under addition, it follows from the standard theory that there exists a Haar measure $dx$ on $\mathbb Q^n_p$, which is unique up to positive constant multiple and is translation invariant. This measure is unique by normalizing $dx$ such that
\[
\int\limits\limits_{B_0}dx=|B_0|=1,
\]
where $|B|$ denotes the Haar measure of a measurable subset $B$ of $\mathbb Q^n_p$. By simple calculation, it is easy to obtain that $|B_\alpha(a)|=p^{n\alpha}, |S_\alpha(a)|=p^{n\alpha}(1-p^{-n})\simeq p^{n\alpha}$, for any $a\in\mathbb Q^n_p$.  For  $f\in L^1_{\text{loc}}(\mathbb Q^n_p)$, we have
$$\int_{\mathbb Q^n_p}f(x)dx=\lim_{\alpha\rightarrow +\infty}\int_{B_\alpha} f(x)dx=\lim_{\alpha\rightarrow +\infty}\sum_{-\infty<\gamma\leq\alpha}\int_{S_\gamma} f(x)dx.$$
In particular, if $f\in L^1(\mathbb Q^n_p)$, we can write 
$$\int_{\mathbb Q^n_p}f(x)dx=\sum_{\alpha=-\infty}^{+\infty}\int_{S_\alpha} f(x)dx,$$
and 
$$\int_{\mathbb Q^n_p}f(tx)dx=\frac{1}{|t|_p^n}\int_{\mathbb Q^n_p} f(x)dx,$$
where $t\in\mathbb Q_p\setminus\{0\}$. For a more complete introduction to the $p$-adic analysis, we refer the readers to \cite{Khrennikov1, Vladimirov2} and the references therein.

Let $\omega(x)$ be a weighted function, that is a non-negative locally integrable measurable function on $\mathbb Q^n_p$. The weighted Lebesgue space $L^q_\omega(\mathbb Q^n_p) \left(0 < q < \infty\right)$ is defined to be the space of all measurable functions $f$ on $\mathbb Q^n_p$ such that
\[
\|f\|_{L^q_\omega(\mathbb Q^n_p)}=\Big(\int_{\mathbb Q^n_p}|f(x)|^q\omega(x)dx\Big)^{1/q}<\infty.
\]
The space $L^q_{\omega, \,\rm loc}(\mathbb Q^n_p)$ is defined as the set of all measurable functions $f$ on $\mathbb Q^n_ p$ satisfying $\int_{K}|f(x)|^q\omega(x)dx<\infty$, for any compact subset $K$ of $\mathbb Q^n_p$. The space $L^q_{\omega,\rm loc}(\mathbb Q^n_p\setminus\{0\})$ is also defined in a similar way as the space $L^q_{\omega,\rm loc}(\mathbb Q^n_p)$.
\vskip 5pt
Throught the whole paper, we denote by $C$ a positive geometric constant that is independent of the main parameters, but can change from line to line. We also write $a \lesssim b$ to mean that there is a positive constant $C$, independent of the main parameters, such that $a\leq Cb$. The symbol $f\simeq g$ means that $f$ is equivalent to $g$ (i.e.~$C^{-1}f\leq g\leq Cf$). For any real number $\ell>1$, denote by $\ell'$ conjugate real number of $\ell$, i.e. $\frac{1}{\ell}+\frac{1}{\ell'}=1$. Denote $\omega(B)^{\lambda}=\left(\int_B{\omega(x)}dx\right)^{\lambda}$, for $\lambda\in\mathbb R$. Remark that if $\omega(x) = |x|_p^{\alpha}$ for $\alpha >-n$, then we have
\begin{align}\label{espower}
\omega(B_{\gamma})=\int_{B_{\gamma}}|x|_p^{\alpha}dx=\sum\limits_{k\leq \gamma} \int_{S_k}p^{k\alpha}dx=\sum\limits_{k\leq \gamma} p^{k(\alpha+n)}(1-p^{-n})\simeq p^{\gamma(\alpha+n)}.
\end{align}

\vskip 5pt
Next, let us give the definition of weighted $\lambda$-central Morrey spaces on $p$-adic numbers field as follows. 
\begin{definition} Let $\lambda\in\mathbb R$ and $ 1<q<\infty$.
The  weighted $\lambda$-central Morrey $p$-adic spaces ${\mathop{B}\limits^.}^{q,\lambda}_{\omega}(\mathbb Q^n_p)$ consists of all Haar measurable functions $f\in L^q_{\omega,\rm loc}(\mathbb Q_p^n)$  satisfying $\|f\|_{{\mathop{B}\limits^.}^{q,\lambda}_{\omega}(\mathbb Q^n_p)}<\infty$,
where
\begin{equation}
\|f\|_{{\mathop{B}\limits^.}^{q,\lambda}_{\omega}(\mathbb Q^n_p)}=\mathop{\rm sup}\limits_{\gamma\in \mathbb Z}\Big(\dfrac{1}{\omega(B_\gamma)^{1+\lambda q}}\int_{B_\gamma}|f(x)|^q\omega(x)dx\Big)^{1/q}.
\end{equation}
Remark that ${\mathop{B}\limits^.}^{q,\lambda}_{\omega}(\mathbb Q^n_p)$ is a Banach space and reduces to $\{0\}$ when $\lambda<-\frac{1}{q}$.
\end{definition}

Let us recall the definition of the weighted central BMO $p$-adic space.
\begin{definition}
Let $1\leq q<\infty$ and $\omega$ be a weight function. The weighted central bounded mean oscillation space ${{CMO}}^q_\omega(\mathbb Q^n_p)$ is defined as the set of all functions $f\in L^q_{\omega,\rm loc}(\mathbb Q^n_p)$ such that
\begin{equation}
\big\|f\big\|_{{{CMO}}^q_\omega(\mathbb Q^n_p)}=\mathop {\rm sup}\limits_{\gamma\in \mathbb Z}\Big( \frac{1}{\omega(B_\gamma)}\int\limits_{B_\gamma}{|f(x)-f_{B_\gamma}|^q\omega(x)dx}\Big)^{\frac{1}{q}}<\infty,
\end{equation}
where 
\[
f_{B_\gamma}=\frac{1}{|B_\gamma|}\int\limits_{B_\gamma}{f(x)dx}.
\]
\end{definition} 
\vskip 5pt
The theory of $A_\ell$ weight was first introduced by Benjamin Muckenhoupt on the Euclidean spaces in order to characterise the boundedness of Hardy-Littlewood maximal functions on the weighted $L^\ell$ spaces (see \cite{Muckenhoupt1972}). For $A_\ell$ weights on the $p$-adic fields,  more generally, on the local fields or homogeneous type spaces,  one can refer to \cite{chuonghung, HCE2012} for more details. Let us now recall the definition of $A_\ell$ weights.
\begin{definition} Let $1 < \ell < \infty$. It is said that a nonnegative locally integrable function $\omega \in A_\ell(\mathbb Q^n_p)$ if there exists a constant $C$ such that for all balls $B$, we have
$$\Big(\dfrac{1}{|B|}\int_{B}\omega(x)dx\Big)\Big(\dfrac{1}{|B|}\int_{B}\omega(x)^{-1/(\ell-1)}dx\Big)^{\ell-1}\leq C.$$
It is said that a weight $\omega\in A_1(\mathbb Q^n_p)$ if there is a constant $C$ such that for all balls $B$, we get
$$\dfrac{1}{|B|}\int_{B}\omega(x)dx\leq C\mathop{\rm essinf}\limits_{x\in B}\omega(x).$$
\end{definition}
We denote by $A_{\infty}(\mathbb Q^n_p) = \bigcup\limits_{1\leq \ell<\infty}A_\ell(\mathbb Q^n_p)$. Let us give the following standard result related to the Muckenhoupt weights.
\begin{proposition}
\begin{enumerate}
\item[\rm (i)] $A_\ell(\mathbb Q^n_p)\subsetneq A_q(\mathbb Q^n_p)$, for $1\leq  \ell < q < \infty$.
\item[\rm (ii)] If $\omega\in A_\ell(\mathbb Q^n_p)$ for  $1 < \ell < \infty$, then there is an $\varepsilon > 0$ such that $\ell-\varepsilon > 1$ and $\omega\in A_{\ell-\varepsilon}(\mathbb Q^n_p)$.
\end{enumerate}
\end{proposition}

A closing relation to $A_{\infty}(\mathbb Q^n_p)$ is the reverse H\"{o}lder condition. If there exist $r > 1$ and a fixed constant $C$ such that
$$\Big(\dfrac{1}{|B|}\int_{B}\omega(x)^rdx\Big)^{1/r}\leq \dfrac{C}{|B|}\int_{B}\omega(x)dx,$$
for all balls $B \subset\mathbb Q^n_p$, we then say that $\omega$ satisfies the reverse H\"{o}lder condition of order $r$ and write $\omega\in RH_r(\mathbb Q^n_p)$. According to Theorem 19 and Corollary 21 in \cite{IMS2015}, $\omega\in A_{\infty}(\mathbb Q^n_p)$ if and
only if there exists some $r > 1$ such that $\omega\in RH_r(\mathbb Q^n_p)$. Moreover, if $\omega\in RH_r(\mathbb Q^n_p)$, $r > 1$, then $\omega\in RH_{r+\varepsilon}(\mathbb Q^n_p)$ for some $\varepsilon > 0$. We thus write $r_\omega = {\rm sup}\{r > 1: \omega\in RH_r(\mathbb Q^n_p)\}$ to denote the critical index of $\omega$ for the reverse H\"{o}lder condition.

An important example of $A_\ell(\mathbb Q^n_p)$ weight is the power function $|x|_p^{\alpha}$. By the similar arguments as Propositions 1.4.3 and 1.4.4 in \cite{LDY2007}, we obtain the following properties of power weights.
\begin{proposition}\label{pro_power}
Let $x\in\mathbb Q^n_p$. Then, we have
\begin{itemize}
\item[(i)] $|x|^{\alpha}_p\in A_1(\mathbb Q^n_p)$ if and only if $-n< \alpha\leq 0$;
\item[(ii)] $|x|^{\alpha}_p\in A_\ell(\mathbb Q^n_p)$ for $1 < \ell< \infty$, if and only if $-n < \alpha < n(\ell-1)$.
\end{itemize}
\end{proposition}
Let us give the following standard characterization of $A_\ell$ weights which it is proved in the similar way as the real setting (see \cite{Grafakos, Stein} for more details).
\begin{proposition}\label{pro2.3DFan}
Let $\omega\in A_\ell(\mathbb Q^n_p) \cap RH_r(\mathbb Q^n_p)$, $\ell\geq 1$ and $r > 1$. Then,  there exist constants $C_1, C_2 > 0$ such that
$$
C_1\left(\dfrac{|E|}{|B|}\right)^\ell\leq \dfrac{\omega(E)}{\omega(B)}\leq C_2\left(\dfrac{|E|}{|B|}\right)^{(r-1)/r}
$$
for any measurable subset $E$ of a ball $B$.
\end{proposition}
\begin{proposition}\label{pro2.4DFan}
If $\omega\in A_\ell(\mathbb Q^n_p)$, $1 \leq \ell < \infty$, then for any $f\in L^1_{\rm loc}(\mathbb Q^n_p)$ and any ball $B \subset \mathbb Q^n_p$, we have
$$
\dfrac{1}{|B|}\int_{B}|f(x)|dx\leq C\left(\dfrac{1}{\omega(B)}\int_{B}|f(x)|^\ell\omega(x)dx\right)^{1/\ell}.
$$
\end{proposition}
Let us recall the definition of the Hardy-Littlewood maximal operator
$$
\mathcal{M}f(x)=\mathop{\rm sup}\limits_{\gamma\in\mathbb Z}\dfrac{1}{p^{n\gamma}}\int_{B_\gamma(x)}|f(y)|dy.
$$
It is useful to remark that the Hardy-Littlewood maximal operator $M$ is bounded on $L^\ell_\omega(\mathbb Q^n_p)$ if and only if $\omega\in A_\ell(\mathbb Q^n_p)$ for all $\ell>1$. Finally, we introduce a new maximal operator which is used in the sequel, that is
$$
\mathcal{M}^{mod}f(x)= \mathop {\rm sup }\limits_{\scriptstyle \,\,\,\,\,\gamma\in\mathbb Z\hfill\atop
\scriptstyle |x|_p\leq p^{\gamma}\hfill}\dfrac{1}{p^{n\gamma}}\int_{B_\gamma(x)}|f(y)|dy.
$$
\section{The main results about the boundness of ${\mathcal{H}}^p_{\Phi,\vec A}$}\label{section3}
Let us now assume that $q$ and $q_i\in [1,\infty)$, $\alpha,\alpha_i$ are real numbers such that $\alpha_i \in (-n,\infty)$, for $i=1,2,...,m$ and 
$$   \frac{1}{{{q_1}}} + \frac{1}{{{q_2}}} + \cdots + \frac{1}{{{q_m}}} = \frac{1}{q}, $$
  $$   \frac{{\alpha_1}}{{{q_{1}}}} + \frac{\alpha _2}{{{q_{2}}}} + \cdots + \frac{\alpha_m}{{{q_{m}}}} = \frac{\alpha}{q}.
$$

In this section, we will  investigate the boundedness of multilinear Hausdorff operators on weighted Lebesgue spaces and weighted central Morrey spaces associated to the case  of matrices having the important property as follows: there exists ${\nu_{\vec{A} }}\in\mathbb N$ such that 
\begin{equation}\label{DK1}
\left\| {{A_i}(y)} \right\|_p.\left\| {A_i^{ - 1}(y)} \right\|_p \leq p^{\nu_{\vec{A} }}, \,\,\,{\text{\rm for all  }}\,i=1,...,m,
\end{equation}
for almost everywhere $y \in \mathbb Q^n_p$. Thus, by the property of invertible matrice, it is easy  to show that
\begin{equation}\label{DK2}
\left\| A_i(y) \right\|_p^\sigma \lesssim \left\| A_i^{-1}(y)\right\|_p^{-\sigma},\,\textit{\rm for all }\,\sigma \,\in \mathbb R,
\end{equation}
and
\begin{equation}\label{DK3}
|A_i(y)x|^\sigma_p \gtrsim \left\| A_i^{-1}(y)\right\|^{-\sigma}_p.|x|^{\sigma}_p, \,\textit{\rm for all }\, \sigma\in\mathbb R,x\in \mathbb Q^n_p\setminus\{0\}.
\end{equation}
Our first main result is the following.
\begin{theorem}\label{theo-HfLq}
Let $\omega_1(x)=|x|_p^{\alpha_1}, ..., \omega_m(x)=|x|_p^{\alpha_m}$ and $\omega(x)=|x|_p^{\alpha}$. Then, $ {\mathcal{H}}^p_{\Phi,\vec A}$ is bounded from $L^{q_1}_{\omega_1}(\mathbb Q_p^n)\times \cdots\times L^{q_m}_{\omega_m}(\mathbb Q^n_p)$ to $L^q_\omega(\mathbb Q_p^n)$ if and only if
$$
\mathcal C_1=\int_{\mathbb Q^n_p}\dfrac{\Phi(y)}{|y|_p^n}\prod\limits_{i=1}^m \|A_i^{-1}(y)\|_p^{\frac{\alpha_i+n}{q_i}}dy<\infty.
$$
Furthermore, ${\big\|{\mathcal H}^p_{\Phi,\vec A}\big\|_{L^{q_1}_{\omega_1}(\mathbb Q_p^n)\times \cdots\times L^{q_m}_{\omega_m}(\mathbb Q^n_p)\to L^q_\omega(\mathbb Q_p^n)}\simeq \mathcal C_{1}}.$
\end{theorem}
\begin{proof}
Firstly, we will prove for the sufficient condition of the theorem. By applying the Minkowski inequality and the H\"{o}lder inequality, we have
\begin{align}
\|\mathcal H^p_{\Phi,\vec A}(\vec f)\|_{L^q_\omega(\mathbb Q_p^n)} &=\Big(\int_{\mathbb Q^n_p}\Big|\int_{\mathbb Q^n_p} \dfrac{\Phi(y)}{|y|^n_p}\prod\limits_{i=1}^{m} f_i(A_i(y)x)dy\Big|^q\omega(x)dx\Big)^{\frac{1}{q}}\nonumber
\\
&\leq \int_{\mathbb Q^n_p}\dfrac{\Phi(y)}{|y|_p^n}\prod\limits_{i=1}^m\|f_i(A_i(y).)\|_{L^{q_i}_{\omega_i}(\mathbb Q^n_p)}dy.
\nonumber
\end{align}
By making the change of variables, we get
\begin{align}
\|f_i(A_i(y).)\|_{L^{q_i}_{\omega_i}(\mathbb Q^n_p)}&=\Big(\int_{\mathbb Q^n_p}|f_i(z)|^{q_i}|A_i^{-1}(y)z|_p^{\alpha_i}.|{\rm det}
A_i^{-1}(y)|_pdz\Big)^{\frac{1}{q_i}}\nonumber
\\
&\leq{\rm max}\{\|A_i^{-1}(y)\|_p^{\alpha_i}, \|A_i(y)\|_p^{-\alpha_i}\}^{\frac{1}{q_i}}.|{\rm det} A_i^{-1}(y)|_p^{\frac{1}{q_i}}\|f_i\|_{L^{q_i}_{\omega_i}(\mathbb Q^n_p)}.
\nonumber
\end{align}
Thus, 
\begin{align}\label{HphiALq}
&\|\mathcal H^p_{\Phi,\vec A}(\vec f)\|_{L^q_\omega(\mathbb Q_p^n)}\nonumber
\\
&\leq \Big(\int_{\mathbb Q^n_p}\dfrac{\Phi(y)}{|y|_p^n}\prod\limits_{i=1}^m {\rm max}\{\|A_i^{-1}(y)\|_p^{\alpha_i}, \|A_i(y)\|_p^{-\alpha_i}\}^{\frac{1}{q_i}}.|{\rm det} A_i^{-1}(y)|_p^{\frac{1}{q_i}}dy\Big)\prod\limits_{i=1}^m \|f_i\|_{L^{q_i}_{\omega_i}(\mathbb Q^n_p)}.
\end{align}
Note that, by (\ref{|det|}) and (\ref{DK2}), we have
\begin{align}\label{Aidet}
{\rm max}\{\|A_i^{-1}(y)\|_p^{\alpha_i}, \|A_i(y)\|_p^{-\alpha_i}\}^{\frac{1}{q_i}}.|{\rm det} A_i^{-1}(y)|_p^{\frac{1}{q_i}}\lesssim \|A_i^{-1}(y)\|_p^{\frac{(\alpha_i+n)}{q_i}}.
\end{align}
This shows that
$$
\|\mathcal H^p_{\Phi,\vec A}(\vec f)\|_{L^q_\omega(\mathbb Q_p^n)}\lesssim \mathcal C_1. \prod\limits_{i=1}^m\|f_i\|_{L^{q_i}_{\omega_i}(\mathbb Q^n_p)}.
$$

Next, to prove the necessary condition of this theorem, for $i=1, ..., m$ and $r\in\mathbb Z^{+}$, let us now take
\begin{align}
f_{i,r}(x) = \left\{ \begin{array}{l}
0,\,\,\,\,\,\,\,\,\,\,\,\,\,\,\,\,\,\,\,\,\,\,\,\,\,\,\,\,\,\,\textit{\rm if}\,|x|_p\leq p^{-\nu_{\vec A}-1},
\\
|x|_p^{-\frac{n}{q_i}-\frac{\alpha_i}{q_i}-p^{-r}},\,\textit{\rm otherwise.}
\end{array} \right.
\nonumber
\end{align}
By a simple calculation, we have
\begin{align}\label{fir}
\|f_{i,r}\|_{L^{q_i}_{\omega_i}(\mathbb Q^n_p)}&=\Big(\int_{\mathbb Q^n_p}|x|_p^{-n-\alpha_i-q_i p^{-r}}\chi_{B^{c}_{-\nu_{\vec A}-1}}(x).|x|_p^{\alpha_i}dx\Big)^{\frac{1}{q_i}}=\Big(\sum\limits_{k\geq -\nu_{\vec A}}\,\int_{S_k}p^{k(-n- q_ip^{-r})}dx\Big)^{\frac{1}{q_i}}\nonumber
\\
&\simeq\Big(\sum\limits_{k\geq -\nu_{\vec A}} p^{k(-n- q_ip^{-r})} p^{kn}\Big)^{\frac{1}{q_i}}=\Big(\sum\limits_{k\geq -\nu_{\vec A}} p^{-kq_ip^{-r}}\Big)^{\frac{1}{q_i}}=\dfrac{p^{\nu_{\vec A}. p^{-r}}}{(1-p^{-q_ip^{-r}})^{\frac{1}{q_i}}}.
\end{align}
Next, we define two sets as follows
$$
S_{x}=\bigcap_{i=1}^m \Big\{y\in\mathbb Q_p^n: |A_i(y)x|_p\geq p^{-\nu_{\vec A}}\Big\},
$$
and
$$
U_r=\Big\{y\in\mathbb Q_n^p: \|A_i(y)\|_p\geq p^{-r},\,\textit{\rm for all}\,i=1,...,m \Big\}.
$$
From this, we derive
\begin{equation}\label{UrSx}
 U_r \subset S_x,\,\,\text{for all}\,\, x\in\mathbb Q^n_p\setminus B_{r-1}.
\end{equation}
In fact, by letting $y\in U_r$, we have $\big\|A_i(y)\big\|_p.|x|_p\geq 1,\,\text{for all}\,\, x\in\mathbb Q^n_p\setminus B_{r-1}.$ Thus, by applying the condition (\ref{DK1}), one has
 \[
|A_i(y)x|_p\geq \big\|A_i^{-1}(y)\big\|^{-1}_p|x|_p=\dfrac{\|A_i(y)\|_p.|x|_p}{\|A_i^{-1}(y)\|_p.\|A_i(y)\|_p}\geq p^{-\nu_{\vec A}},
\]
which confirms the relation (\ref{UrSx}). 
Now, by taking $x\in\mathbb Q_p^n\setminus B_{r-1}$ and using the relation (\ref{UrSx}), we get
\[
{\mathcal H^p_{\Phi ,\vec A }}(\vec{f)} (x) \geq \int\limits_{S_{x}} {\dfrac{\Phi(y)}{|y|_p^n}} \prod\limits_{i = 1}^m {{{| {{A_i}(y)x}|_p}^{- \frac{n}{q_i}-\frac{\alpha_i}{q_i} -p^{-r}}}} dy \geq \int\limits_{U_r} {\dfrac{\Phi(y)}{|y|_p^n}} \prod\limits_{i = 1}^m {{{| {{A_i}(y)x}|_p}^{- \frac{n}{q_i}-\frac{\alpha_i}{q_i} -p^{-r}}}} dy  .
\]
From this, by (\ref{DK3}), one has
\begin{align}
{\mathcal H^p_{\Phi ,\vec A }}(\vec{f)} (x) &\gtrsim \Big(\int_{U_r}\dfrac{\Phi(y)}{|y|_p^n}\prod\limits_{i=1}^m\|A_i^{-1}(y)\|_p^{\frac{n+\alpha_i}{q_i}+p^{-r}}dy\Big).|x|_p^{-\frac{(n+\alpha)}{q}-mp^{-r}}\chi_{\mathbb Q^n_p\setminus B_{r-1}}(x).\nonumber
\\
&:=p^{rm p^{-r}}\mathcal A_{r}.g(x),\nonumber
\end{align}
where 
$$
\mathcal A_r=\int_{U_r}\dfrac{\Phi(y)}{|y|_p^n}\prod\limits_{i=1}^m\|A_i^{-1}(y)\|_p^{\frac{n+\alpha_i}{q_i}}p^{-rm p^{-r}}.\prod\limits_{i=1}^m {\|A_i^{-1}(y)\|_p}^{p^{-r}}dy
$$
and $ g(x)= |x|_p^{-\frac{(n+\alpha)}{q}-mp^{-r}}.\chi_{\mathbb Q^n_p\setminus B_{r-1}}(x).$
By estimating as (\ref{fir}) above, we also have
$$
\|g\|_{L^{q}_\omega(\mathbb Q^n_p)}\simeq \dfrac{p^{-r mp^{-r}}}{(1-p^{-q m p^{-r}})^{\frac{1}{q}}}.
$$
As a consequence above, by (\ref{fir}), we find that
\begin{align}
\|\mathcal H^p_{\Phi,\vec A}(\vec f)\|_{L^q_\omega(\mathbb Q^n_p)}\gtrsim \mathcal A_r \dfrac{\prod\limits_{i=1}^m \|f_{i,r}\|_{L^{q_i}_{\omega_i}(\mathbb Q^n_p)}}{(1-p^{-q m p^{-r}})^{\frac{1}{q}}\prod\limits_{i=1}^m \dfrac{p^{\nu_{\vec A} p^{-r}}}{(1-p^{-q_ip^{-r}})^{\frac{1}{q_i}}}}:=\mathcal A_r. \mathcal T_r.\prod\limits_{i=1}^m \|f_{i,r}\|_{L^{q_i}_{\omega_i}(\mathbb Q^n_p)},\nonumber
\end{align}
where $\mathcal T_r=\dfrac{\prod\limits_{i=1}^m(1-p^{-q_ip^{-r}})^{\frac{1}{q_i}}}{(1-p^{-qmp^{-r}})^{\frac{1}{q}}p^{m.\nu_{\vec A}.p^{-r}}}$.
Note that from $\dfrac{1}{q_1}+ \cdots +\dfrac{1}{q_m}=\dfrac{1}{q}$, it is clear to obtain that
$$
\mathop{\rm lim}\limits_{r\to +\infty} \mathcal T_r=a>0.
$$
Therefore, because $ {\mathcal{H}}^p_{\Phi,\vec A}$ is bounded from $L^{q_1}_{\omega_1}(\mathbb Q_p^n)\times \cdots\times L^{q_m}_{\omega_m}(\mathbb Q^n_p)$ to $L^q_\omega(\mathbb Q_p^n)$, there exists $M>0$ such that $\mathcal A_r\leq M$, for  $r$ sufficiently big. 
On the other hand, by letting $r$ sufficiently large, $y\in U_r$ and by (\ref{DK1}), we get
$$
p^{-rm p^{-r}}.\prod\limits_{i=1}^m {\|A_i^{-1}(y)\|_p}^{p^{-r}}\lesssim p^{\nu_{\vec A}.m.p^{-r}}\lesssim 1.
$$
From this, by the dominated convergence theorem of Lebesgue, we obtain
$$
\int_{\mathbb Q^n_p}\dfrac{\Phi(y)}{|y|_p^n}\prod\limits_{i=1}^m \|A_i^{-1}(y)\|_p^{\frac{\alpha_i+n}{q_i}}dy<\infty,
$$
which finishes the proof of the theorem.
\end{proof}

\begin{theorem}\label{theo-HfLq*}
Let $1\leq q^*,\zeta<\infty$ and $\omega\in A_{\zeta}$ with the finite critical index $r_{\omega}$ for the reverse H\"{o}lder and $\omega(B_\gamma)\lesssim 1$, for all $\gamma\in\mathbb Z$. Assume that $q>q^*\zeta r_{\omega}/(r_{\omega}-1)$, $\delta\in (1,r_\omega)$ and the following condition holds:
\begin{align}
\mathcal C_2&=\int_{\mathbb Q_p^n}\dfrac{\Phi(y)}{|y|_p^n}\prod\limits_{i=1}^m |{\rm det} A_i^{-1}(y)|^{\frac{\zeta}{q_i}}_p\|A_i(y)\|_p^{\frac{\zeta}{q_i}}\times\nonumber
\\
&\times\Big(\chi_{\{\|A_i(y)\|_p\leq 1\}}(y)\|A_i(y)\|_p^{\frac{-n\zeta}{q_i}} +\chi_{\{\|A_i(y)\|_p>1\}(y)}\|A_i(y)\|_p^{\frac{-n(\delta-1)}{q_i\delta}}\Big)dy<\infty.\nonumber
\end{align}
Then, we have ${\mathcal H}^p_{\Phi,\vec A}$ is bounded from $L^{q_1}_{\omega}(\mathbb Q_p^n)\times \cdots\times L^{q_m}_{\omega}(\mathbb Q^n_p)$ to $L^{q^*}_\omega(\mathbb Q_p^n)$.
\end{theorem}
\begin{proof}
For any $R\in\mathbb Z$, by the Minkowski inequality, we have
$$
\|\mathcal H^p_{\Phi,\vec A}(\vec f)\|_{L^{q^*}_\omega(B_R)}\leq \int_{\mathbb Q_p^n}\dfrac{\Phi(y)}{|y|_p^n}\Big(\int_{B_R}\prod\limits_{i=1}^m |f_i(A_i(y)x)|^{q^*}\omega(x)dx\Big)^{\frac{1}{q^*}}dy.
$$
From the inequality $q>q^*\zeta r_{\omega}/(r_{\omega}-1)$, there exists $r\in (1,r_\omega)$ such that $q=\zeta q^* r'$. Then, by the H\"{o}lder inequality and the reverse H\"{o}lder condition, we have
\begin{align}
\Big(\int_{B_R}\prod\limits_{i=1}^m |f_i(A_i(y)x)|^{q^*}\omega(x)dx\Big)^{\frac{1}{q^*}}&\leq \Big(\int_{B_R}\prod\limits_{i=1}^m|f_i(A_i(y)x)|^{\frac{q}{\zeta}}dx\Big)^{\frac{\zeta}{q}}.\Big(\int_{B_R}\omega(x)^rdx\Big)^{\frac{1}{rq^*}}\nonumber
\\
&\lesssim \Big(\int_{B_R}\prod\limits_{i=1}^m|f_i(A_i(y)x)|^{\frac{q}{\zeta}}dx\Big)^{\frac{\zeta}{q}}\omega(B_R)^{\frac{1}{q^*}}|B_R|^{-\frac{\zeta}{q}}.\nonumber
\end{align}
Next, by making the H\"{o}lder inequality and the change of variables formula, and applying Proposition \ref{pro2.4DFan}, we have
\begin{align}
&\Big(\int_{B_R}\prod\limits_{i=1}^m|f_i(A_i(y)x)|^{\frac{q}{\zeta}}dx\Big)^{\frac{\zeta}{q}}\leq \prod\limits_{i=1}^m \Big(\int_{B_R}|f_i(A_i(y)x)|^{\frac{q_i}{\zeta}}dx\Big)^{\frac{\zeta}{q_i}}\nonumber
\\
&=\prod\limits_{i=1}^m\Big(\int_{A_i(y)B_R}|f_i(z)|^{\frac{q_i}{\zeta}}|{\rm det} A_i^{-1}(y)|_pdz\Big)^{\frac{\zeta}{q_i}}\leq \prod\limits_{i=1}^m|{\rm det} A_i^{-1}(y)|_p^{\frac{\zeta}{q_i}}\|f_i\|_{L^{\frac{q_i}{\zeta}}(B_{R+k_{A_i}})}\nonumber
\\
&\lesssim \prod\limits_{i=1}^m|{\rm det} A_i^{-1}(y)|_p^{\frac{\zeta}{q_i}} |B_{R+k_{A_i}}|^{\frac{\zeta}{q_i}}\omega(B_{R+k_{A_i}})^{-\frac{1}{q_i}}\|f_i\|_{L^{q_i}_{\omega}(B_{R+k_{A_i}})},\nonumber
\end{align}
where $k_{A_i}(y)={\rm log_p}\|A_i(y)\|_p$.
Thus, by $\frac{|B_{R+k_{A_i}}|}{|B_R|}\simeq \frac{p^{(R+k_{A_i})n}}{p^{Rn}}=\|A_i(y)\|_p^n$, we infer that
\begin{align}\label{HfLq*}
&\|\mathcal H^p_{\Phi,\vec A}(\vec f)\|_{L^{q^*}_\omega(B_R)}\nonumber
\\
&\lesssim \omega(B_R)^{\frac{1}{q^*}}\int_{\mathbb Q^n_p}\dfrac{\Phi(y)}{|y|_p^n} \prod\limits_{i=1}^m|{\rm det} A_i^{-1}(y)|_p^{\frac{\zeta}{q_i}}\Big(\frac{|B_{R+k_{A_i}}|}{|B_R|}\Big)^{\frac{\zeta}{q_i}}\omega(B_{R+k_{A_i}})^{-\frac{1}{q_i}}\|f_i\|_{L^{q_i}_{\omega}(B_{R+k_{A_i}})}dy
\nonumber
\\
&\lesssim \omega(B_R)^{\frac{1}{q^*}}\int_{\mathbb Q^n_p}\dfrac{\Phi(y)}{|y|_p^n} \prod\limits_{i=1}^m|{\rm det} A_i^{-1}(y)|_p^{\frac{\zeta}{q_i}}\|A_i(y)\|_p^{\frac{\zeta n}{q_i}}\omega(B_{R+k_{A_i}})^{-\frac{1}{q_i}}\|f_i\|_{L^{q_i}_{\omega}(B_{R+k_{A_i}})}dy.
\end{align}
On the other hand, by $q>q^*\geq 1$ and $\omega(B_R)\lesssim 1$ for all $R\in\mathbb Z$, we imply that $\omega(B_R)^{\frac{1}{q^*}}\lesssim \omega(B_R)^{\frac{1}{q}}$. Hence, by (\ref{HfLq*}), we get
\begin{align}
&\|\mathcal H^p_{\Phi,\vec A}(\vec f)\|_{L^{q^*}_\omega(B_R)}\nonumber
\\
&\lesssim \Big(\int_{\mathbb Q^n_p}\dfrac{\Phi(y)}{|y|_p^n} \prod\limits_{i=1}^m|{\rm det} A_i^{-1}(y)|_p^{\frac{\zeta}{q_i}}\|A_i(y)\|_p^{\frac{\zeta n}{q_i}}\Big(\frac{\omega(B_R)}{\omega(B_{R+k_{A_i}})}\Big)^{\frac{1}{q_i}}dy\Big)\prod\limits_{i=1}^m\|f_i\|_{L^{q_i}_{\omega}(\mathbb Q^n_p)}.\nonumber
\end{align}
Next, for $i=1, ..., m$, by using Proposition \ref{pro2.3DFan}, we have
\begin{align}\label{omegaB_Morrey-Leb}
\Big(\frac{\omega(B_R)}{\omega(B_{R+k_{A_i}})}\Big)^{\frac{1}{q_i}}\lesssim \left\{ \begin{array}{l}
\Big(\dfrac{|B_R|}{|B_{R+k_{A_i}}|}\Big)^{\frac{\zeta}{q_i}}\lesssim p^{\frac{(R-R-k_{A_i})n\zeta}{q_i}}=\|A_i(y)\|_p^{\frac{-n\zeta}{q_i}},\\
\textit{\rm if}\, \|A_i(y)\|_p\leq 1,
\\
\\
\Big(\dfrac{|B_R|}{|B_{R+k_{A_i}}|}\Big)^{\frac{(\delta-1)}{q_i\delta}}\lesssim p^{\frac{(R-R-k_{A_i})n(\delta-1)}{q_i\delta}}=\|A_i(y)\|_p^{\frac{-n(\delta-1)}{q_i\delta}},
\\
\textit{\rm otherwise}.
\end{array} \right.
\end{align}
Hence, by letting $R\to +\infty$ and applying the monotone convergence theorem of Lebesgue, we obtain that
$$
\|\mathcal H^p_{\Phi,\vec A}(\vec f)\|_{L^{q^*}_\omega(\mathbb Q^n_p)}\lesssim \mathcal C_2.\prod\limits_{i=1}^m\|f_i\|_{L^{q_i}_{\omega}(\mathbb Q^n_p)},
$$
which completes the proof of the theorem.
\end{proof}
\begin{theorem}\label{theo-HfMorrey}
Let $\omega_i,\omega$ be as Theorem \ref{theo-HfLq} and $\lambda_i\in (\frac{-1}{q_i},0)$ for all $i=1, ... m$. Assume that 
\begin{align}\label{lambdaMorrey}
(\alpha+n)\lambda=(\alpha_1+n)\lambda_1 + \cdots + (\alpha_m+n)\lambda_m.
\end{align}
Then, $ {\mathcal{H}}^p_{\Phi,\vec A}$ is bounded from ${\mathop{B}\limits^.}^{q_1,\lambda_1}_{\omega_1}(\mathbb Q_p^n)\times \cdots\times {\mathop{B}\limits^.}^{q_m, \lambda_m}_{\omega_m}(\mathbb Q^n_p)$ to ${\mathop{B}\limits^.}^{q,\lambda}_\omega(\mathbb Q_p^n)$ if and only if
$$
\mathcal C_3=\int_{\mathbb Q^n_p}\dfrac{\Phi(y)}{|y|_p^n}\prod\limits_{i=1}^m \|A_i^{-1}(y)\|_p^{-(\alpha_i+n)\lambda_i}dy<\infty.
$$
Furthermore, ${\big\|{\mathcal H}^p_{\Phi,\vec A}\big\|_{{\mathop{B}\limits^.}^{q_1,\lambda_1}_{\omega_1}(\mathbb Q_p^n)\times \cdots\times {\mathop{B}\limits^.}^{q_m, \lambda_m}_{\omega_m}(\mathbb Q^n_p)\to {\mathop{B}\limits^.}^{q,\lambda}_\omega(\mathbb Q_p^n)}\simeq \mathcal C_{3}}.$
\end{theorem}
\begin{proof}
We start with the proof for the sufficient condition of the theorem. For $\gamma\in\mathbb Z$, by estimating as (\ref{HphiALq}) and (\ref{Aidet}) above, we have
\begin{align}
&\|\mathcal H^p_{\Phi,\vec A}(\vec f)\|_{L^q_\omega(B_\gamma)}\nonumber
\\
&\leq \int_{\mathbb Q^n_p}\dfrac{\Phi(y)}{|y|_p^n}\prod\limits_{i=1}^m{\rm max}\{\|A_i^{-1}(y)\|_p^{\alpha_i}, \|A_i(y)\|_p^{-\alpha_i}\}^{\frac{1}{q_i}}.|{\rm det} A_i^{-1}(y)|_p^{\frac{1}{q_i}}\prod\limits_{i=1}^m \|f_i\|_{L^{q_i}_{\omega_i}(A_i(y)B_{\gamma})}dy.\nonumber 
\\
&\lesssim \int_{\mathbb Q^n_p}\dfrac{\Phi(y)}{|y|_p^n}\prod\limits_{i=1}^m\|A_i^{-1}(y)\|_p^{\frac{(\alpha_i+n)}{q_i}}\prod\limits_{i=1}^m \|f_i\|_{L^{q_i}_{\omega_i}(B_{\gamma+k_{A_i}})}dy.\nonumber 
\end{align}
This implies that
\begin{align}\label{esHfMorrey-theo3.3}
\dfrac{1}{\omega(B_\gamma)^{\frac{1}{q}+\lambda}}\|\mathcal H^p_{\Phi,\vec A}(\vec f)\|_{L^q_\omega(B_\gamma)}&\lesssim \int_{\mathbb Q^n_p}\dfrac{\Phi(y)}{|y|_p^n}\Big(\prod\limits_{i=1}^m\|A_i^{-1}(y)\|_p^{\frac{(\alpha_i+n)}{q_i}}\Big)\mathcal B_i(y)\times\nonumber
\\
&\times\Big(\prod\limits_{i=1}^m \dfrac{1}{\omega_i(B_{\gamma+k_{A_i}})^{\frac{1}{q_i}+\lambda_i}}\|f_i\|_{L^{q_i}_{\omega_i}(B_{\gamma+k_{A_i}})}\Big)dy, 
\end{align}
where $\mathcal B_i(y)= \dfrac{\prod\limits_{i=1}^m\omega_i(B_{\gamma+k_{A_i}})^{\frac{1}{q_i}+\lambda_i}}{\omega(B_{\gamma})^{\frac{1}{q}+\lambda}}.$
\\
On the other hand, by hypothesis (\ref{lambdaMorrey}), we immediately get
$$
\sum\limits_{i=1}^m \big(\alpha_i + n\big)\big(\frac{1}{q_i}+\lambda_i\big)= \big(\alpha+n\big)\big(\frac{1}{q}+\lambda\big).
$$
Consequently, by the estimation (\ref{espower}) and (\ref{DK1}), we have
\begin{align}
\mathcal B_i(y)&\lesssim \dfrac{p^{\sum\limits_{i=1}^m (\gamma+k_{A_i})(\alpha_i+n)(\frac{1}{q_i}+\lambda_i)}}{p^{\gamma(\alpha+n)(\frac{1}{q}+\lambda)}}= \dfrac{p^{\sum\limits_{i=1}^m \gamma(\alpha_i+n)(\frac{1}{q_i}+\lambda_i)}p^{\sum\limits_{i=1}^m k_{A_i}(\alpha_i+n)(\frac{1}{q_i}+\lambda_i)}}{p^{\gamma(\alpha+n)(\frac{1}{q}+\lambda)}}\nonumber
\\
&=\prod\limits_{i=1}^m\|A_i(y)\|_p^{(\alpha_i+n)(\frac{1}{q_i}+\lambda_i)}\lesssim \prod\limits_{i=1}^m\|A_i^{-1}(y)\|_p^{-(\alpha_i+n)(\frac{1}{q_i}+\lambda_i)}.\nonumber
\end{align}
Hence, by (\ref{esHfMorrey-theo3.3}), one has
$$
\|\mathcal H^p_{\Phi,\vec A}(\vec f)\|_{{\mathop{B}\limits^.}^{q,\lambda}_{\omega}(\mathbb Q_p^n)}\lesssim \mathcal C_3 \prod\limits_{i=1}^m \|f_i\|_{{\mathop{B}\limits^.}^{q_i,\lambda_i}_{\omega_i}(\mathbb Q_p^n)}.
$$
Conversely, suppose that $ {\mathcal{H}}^p_{\Phi,\vec A}$ is bounded from ${\mathop{B}\limits^.}^{q_1,\lambda_1}_{\omega_1}(\mathbb Q_p^n)\times \cdots\times {\mathop{B}\limits^.}^{q_m, \lambda_m}_{\omega_m}(\mathbb Q^n_p)$ to ${\mathop{B}\limits^.}^{q,\lambda}_\omega(\mathbb Q_p^n)$. For $i=1, ..., m$, let us choose the functions as follows
\begin{align}
f_i(x)=|x|_p^{(\alpha_i+n)\lambda_i}.\nonumber
\end{align}
Then, by (\ref{espower}), it is not difficult to show that
\begin{align*}
\|f_i\|_{{\mathop{B}\limits^.}^{q_i,\lambda_i}_{\omega_i}(\mathbb Q^n_p)}&=\mathop{\rm sup}\limits_{\gamma\in\mathbb Z}\dfrac{1}{\omega_i(B_\gamma)^{\frac{1}{q_i}+\lambda_i}}\Big(\int_{B_{\gamma}}|x|_p^{(\alpha_i+n)\lambda_iq_i+\alpha_i}dx\Big)^{\frac{1}{q_i}}\nonumber
\\
&\simeq \mathop{\rm sup}\limits_{\gamma\in\mathbb Z}\dfrac{p^{\gamma(\,(\alpha_i+n)\lambda_iq_i+\alpha_i +n)\frac{1}{q_i} }}{p^{\gamma(\alpha_i+n)(\frac{1}{q_i}+\lambda_i)}}=1,
\end{align*}
and similarly, we also have
\begin{align}\label{cal-fi-Morrey}
\|\,|\cdot|_p^{(\alpha+n)\lambda}\|_{{\mathop{B}\limits^.}^{q,\lambda}_{\omega}(\mathbb Q^n_p)}\simeq 1.
\end{align}
Next, by choosing $ f_i$'s and using (\ref{DK3}) and (\ref{lambdaMorrey}), we have
\begin{align}
\mathcal H^p_{\Phi,\vec A}(\vec f)(x)&=\int_{\mathbb Q^n_p}\dfrac{\Phi(y)}{|y|_p^n}\prod\limits_{i=1}^m |A_i(y)x|_p^{(\alpha_i+n)\lambda_i}dy\nonumber
\\
&\gtrsim \int_{\mathbb Q^n_p}\dfrac{\Phi(y)}{|y|_p^n}\prod\limits_{i=1}^m \|A_i^{-1}(y)\|_p^{-(\alpha_i+n)\lambda_i}.|x|_p^{(\alpha_i+n)\lambda_i}dy=\mathcal C_3. |x|_p^{(\alpha+n)\lambda}.\nonumber
\end{align}
Thus, by (\ref{cal-fi-Morrey}), it follows that
\begin{align}
\|\mathcal H^p_{\Phi,\vec A}(\vec f)\|_{{\mathop{B}\limits^.}^{q,\lambda}_{\omega}(\mathbb Q^n_p)}\gtrsim \mathcal C_3.\|\,|\cdot|_p^{(\alpha+n)\lambda}\|_{{\mathop{B}\limits^.}^{q,\lambda}_{\omega}(\mathbb Q^n_p)}\gtrsim \mathcal C_3\prod\limits_{i=1}^m \|f_i\|_{{\mathop{B}\limits^.}^{q_i,\lambda_i}_{\omega_i}(\mathbb Q^n_p)}.\nonumber
\end{align}
This gives that $\mathcal C_3<\infty$. Hence, the theorem is completely proved.
\end{proof}
\begin{theorem}\label{theo-HfMorreyq*}
Let $1\leq q^*,\zeta<\infty$, $\lambda_i\in (-\frac{1}{q_i},0)$, for all $i=1, ...,m$ and $\omega\in A_{\zeta}$ with the finite critical index $r_{\omega}$ for the reverse H\"{o}lder. Assume that $q>q^*\zeta r_{\omega}/(r_{\omega}-1)$, $\delta\in (1,r_\omega)$ and the following two conditions are true:
\begin{align}\label{lambdaMorreyAp}
\lambda=\lambda_1 + \cdots +\lambda_m.
\end{align}
\begin{align}
\mathcal C_4&=\int_{\mathbb Q_p^n}\dfrac{\Phi(y)}{|y|_p^n}\prod\limits_{i=1}^m |{\rm det} A_i^{-1}(y)|^{\frac{\zeta}{q_i}}_p\|A_i(y)\|_p^{\frac{\zeta}{q_i}}\times\nonumber
\\
&\times\Big(\chi_{\{\|A_i(y)\|_p\leq 1\}}(y)\|A_i(y)\|_p^{n\zeta\lambda_i} +\chi_{\{\|A_i(y)\|_p>1\}(y)}\|A_i(y)\|_p^{\frac{n\lambda_i(\delta-1)}{\delta}}\Big)dy<\infty.\nonumber
\end{align}
Then, ${\mathcal H}^p_{\Phi,\vec A}$ is bounded from ${\mathop{B}\limits^.}^{q_1,\lambda_1}_{\omega}(\mathbb Q_p^n)\times \cdots\times {\mathop{B}\limits^.}^{q_m, \lambda_m}_{\omega}(\mathbb Q^n_p)$ to ${\mathop{B}\limits^.}^{q^*,\lambda}_\omega(\mathbb Q_p^n)$ .
\end{theorem}
\begin{proof}
For $\gamma\in\mathbb Z $, by estimating as (\ref{HfLq*}) above and using the relation (\ref{lambdaMorreyAp}), we obtain that
\begin{align}
&\dfrac{1}{\omega(B_\gamma)^{\frac{1}{q^*}+\lambda}}\|\mathcal H^p_{\Phi,\vec A}(\vec f)\|_{L^{q^*}_\omega(B_\gamma)}\nonumber
\\
&\lesssim \int_{\mathbb Q_p^n}\dfrac{\Phi(y)}{|y|_p^n} \prod\limits_{i=1}^m|{\rm det} A_i^{-1}(y)|_p^{\frac{\zeta}{q_i}}\|A_i(y)\|_p^{\frac{\zeta n}{q_i}}\Big(\dfrac{\omega(B_{\gamma+k_{A_i}})}{\omega(B_{\gamma})}\Big)^{\lambda_i}\dfrac{1}{\omega(B_{\gamma+k_{A_i}})^{\frac{1}{q_i}+\lambda_i}}\|f_i\|_{L^{q_i}_{\omega}(B_{\gamma+k_{A_i}})}dy\nonumber
\\
&\lesssim \Big(\int_{\mathbb Q_p^n}\dfrac{\Phi(y)}{|y|_p^n} \prod\limits_{i=1}^m|{\rm det} A_i^{-1}(y)|_p^{\frac{\zeta}{q_i}}\|A_i(y)\|_p^{\frac{\zeta n}{q_i}}\Big(\dfrac{\omega(B_{\gamma+k_{A_i}})}{\omega(B_{\gamma})}\Big)^{\lambda_i}dy\Big)\prod\limits_{i=1}^m\|f_i\|_{{\mathop{B}\limits^.}^{q_i,\lambda_i}_{\omega}(\mathbb Q_p^n)}.\nonumber
\end{align}
In addition, for $i= 1, ..., m$, by making Proposition \ref{pro2.3DFan} again and $\lambda_i<0$, we infer
\begin{align}\label{theo3.4-omega}
\Big(\frac{\omega(B_{\gamma+k_{A_i}})}{\omega(B_{\gamma})}\Big)^{\lambda_i}\lesssim \left\{ \begin{array}{l}
\Big(\dfrac{|B_{\gamma+k_{A_i}}|}{|B_{\gamma}|}\Big)^{\zeta\lambda_i}\lesssim p^{(\gamma+k_{A_i}-\gamma)n\zeta\lambda_i}=\|A_i(y)\|_p^{n\zeta\lambda_i},\\
\textit{\rm if}\, \|A_i(y)\|_p\leq 1,
\\
\\
\Big(\dfrac{|B_{\gamma+k_{A_i}}|}{|B_{\gamma}|}\Big)^{\frac{\lambda_i(\delta-1)}{\delta}}\lesssim p^{\frac{(\gamma+k_{A_i}-\gamma)n\lambda_i(\delta-1)}{\delta}}=\|A_i(y)\|_p^{\frac{n\lambda_i(\delta-1)}{\delta}},
\\
\textit{\rm otherwise}.
\end{array} \right.
\end{align}
Thus, we have
$
\|\mathcal H^p_{\Phi,\vec A}(\vec f)\|_{{\mathop{B}\limits^.}^{q^*,\lambda}_{\omega}(\mathbb Q_p^n)}\lesssim \mathcal C_4 \prod\limits_{i=1}^m \|f_i\|_{{\mathop{B}\limits^.}^{q_i,\lambda_i}_{\omega}(\mathbb Q_p^n)},
$
which gives that the proof of this theorem is ended.
\end{proof}
\section{The main results about the boundness of ${\mathcal{H}}^p_{\Phi,\vec A,\vec b}$}\label{section4}
Before stating our next results, we introduce some notations which will be used throughout this section. Let $q, q_i\in [1,\infty)$, and let $\alpha$, $\alpha_i, r_i$ be real numbers such that $r_i\in (1,\infty)$, $\alpha_i \in (-n,\frac{nr_i}{r_i'})$, $i=1,2,...,m$. Denote
$$  \left(\frac{1}{{{q_1}}} + \frac{1}{{{q_2}}} + \cdots + \frac{1}{{{q_m}}}\right) + \left(\frac{1}{{{r_1}}} + \frac{1}{{{r_2}}} + \cdots + \frac{1}{{{r_m}}}\right) = \frac{1}{q}, $$
  $$  \left( \frac{{\alpha_1}}{{{q_{1}}}} + \frac{\alpha _2}{{{q_{2}}}} + \cdots + \frac{\alpha_m}{{{q_{m}}}} \right)+\left( \frac{{\alpha_1}}{{{r_{1}}}} + \frac{\alpha _2}{{{r_{2}}}} + \cdots + \frac{\alpha_m}{{{r_{m}}}}\right)= \frac{\alpha}{q}.
$$

\begin{lemma}\label{LemmaCMO1Matrix}
Let $\omega(x)=|x|_p^{\alpha}$, $\omega_i(x)=|x|_p^{\alpha_i}$ and $b_i\in {{CMO}}_{\omega_i}^{r_i}(\mathbb Q^n_p)$, for all $i=1, ...,m.$
Then, for any $\gamma\in\mathbb Z$, we have
\begin{align}
&\|{\mathcal{H}}^p_{\Phi,\vec A,\vec b}(\vec f)\|_{L^q_\omega(B_\gamma)}\lesssim p^{\sum\limits_{i=1}^m \frac{\gamma(n+\alpha_i)}{r_i}}\mathcal B_{\vec r,\vec\omega}.\int_{\mathbb Q^n_p}\dfrac{\Phi(y)}{|y|^n_p}\prod\limits_{i=1}^m\psi_i(y).\mu_i(y).\|f_i\|_{L^{q_i}_{\omega_i}(B_{\gamma+k_{A_i}})}dy,\nonumber
\end{align}
where 
\begin{align}
\psi_i(y)&= 1+\big(\textit{\rm max}\,\lbrace \|A_i^{-1}(y)\|_p^{\alpha_i},\|A_i(y)\|_p^{-\alpha_i}\rbrace |\textit{\rm{det}}A_i^{-1}(y)|_p\big)^{\frac{1}{r_i}}\|A_i(y)\|_p^{\frac{(n+\alpha_i)}{r_i}}\nonumber
\\
&\,\,\,\,\,\,\,\,\,\,\,+ |{\rm log}_p\|A_i(y)\|_p| + 2\dfrac{\|A_i(y)\|_p^n}{|{\rm det} A_i(y)|_p},\nonumber
\\
\mu_i(y)&=\Big(\textit{\rm max}\,\lbrace \|A_i^{-1}(y)\|_p^{\alpha_i},\|A_i(y)\|_p^{-\alpha_i}\rbrace |\textit{\rm{det}}A_i^{-1}(y)|_p\Big)^{\frac{1}{q_i}}\,\textit{\rm and}\,\mathcal B_{\vec r,\vec\omega}=\prod\limits_{i=1}^m\|b_i\|_{{{CMO}}^{r_i}_{\omega_i}(\mathbb Q^n_p)}.\nonumber
\end{align}
\end{lemma}
\begin{proof}
By the Minkowski inequality and the H\"{o}lder inequality, for any $\gamma\in\mathbb Z$, we get
\begin{align}\label{HfMatrix}
\big\|\mathcal{H}_{\Phi,\vec A,\vec b}^{p}(\vec f)\big\|_{L^{q}_{\omega}(B_\gamma)}\lesssim \int\limits_{\mathbb Q^n_p}{\frac{\Phi(y)}{|y|^n_p}}\prod\limits_{i=1}^m\big\|b_i(\cdot)-b_i(A_i(y)\cdot)\big\|_{L^{r_i}_{\omega_i}(B_\gamma)}\big\|f_i(A_i(y)\cdot)\big\|_{L^{q_i}_{\omega_i}(B_\gamma)}dy.
\end{align}
To prove this lemma, we need to show that the following inequality holds 
\begin{align}\label{BDT}
&\big\|b_i(\cdot)-b_i(A_i(y)\cdot)\big\|_{L^{r_i}_{\omega_i}(B_\gamma)}\lesssim p^{\frac{\gamma(\alpha_i+n)}{r_i}}.\psi_i(y).\big\|b_i\big\|_{{ {CMO}}^{r_i}_{\omega_i}(\mathbb Q_p^n)},\,\textit{\rm for all }\,i=1,...,m.
\end{align}
We put $I_{1,i}= \|b_i(\cdot)-b_{i,B_\gamma}\|_{L^{r_i}_{\omega_i}(B_\gamma)},$ $ I_{2,i}=\|b_i(A_i(y)\cdot)-b_{i, A_i(y)B_\gamma}\|_{L^{r_i}_{\omega_i}(B_\gamma)}$ and $ I_{3, i}=\|b_{i,B_\gamma}-b_{i,A_i(y)B_\gamma}\|_{L^{r_i}_{\omega_i}(B_\gamma)}.$ It is obvious that
\begin{equation}\label{BDTphantich}
\big\|b_i(\cdot)-b_i(A_i(y)\cdot)\big\|_{L^{r_i}_{\omega_i}{(B_\gamma)}}\leq I_{1,i} + I_{2,i} + I_{3,i},\,\textit{\rm for all }\,i=1, ...,m.
\end{equation}
By the definition of the space ${{CMO}}^{r_i}_{\omega_i}(\mathbb Q^n_p)$ and the estimation (\ref{espower}), we have
\begin{equation}\label{I1i}
I_{1,i}\leq  {\omega}_i(B_\gamma)^{\frac{1}{r_i}}.\big\|b_i\big\|_{{{CMO}}^{r_i}_{\omega_i}(\mathbb Q^n_p)}\lesssim p^{\frac{\gamma(\alpha_i+n)}{r_i}}.\big\|b_i\big\|_{{{CMO}}^{r_i}_{\omega_i}(\mathbb Q^n_p)}.
\end{equation}
To estimate $I_{2,i}$, we deduce that
\begin{align}\label{decomaI2i}
I_{2,i}&=\Big(\int\limits_{B_\gamma}{|b_i(A_i(y)x)-b_{i,A_i(y)B_\gamma}|^{r_i}\omega_i(x)dx}\Big)^{\frac{1}{r_i}}
\nonumber
\\
&\leq  \omega_i(B_\gamma)^{\frac{1}{r_i}}|b_{i,A_i(y)B_\gamma}-b_{i,B_{\gamma+k_{A_i}}}|+\Big(\int\limits_{B_\gamma}{|b_i(A_i(y)x)-b_{i,B_{\gamma+k_{A_i}}}|^{r_i}\omega_i(x)dx}\Big)^{\frac{1}{r_i}},
\end{align}
where $k_{A_i}(y)= {\rm log_p}\|A_i(y)\|_p$.
Note that, by the formula for change of variables, we get
\begin{align}\label{|AiB|}
|A_i(y)B_\gamma|&=\int_{A_i(y)B_\gamma}dx=\int_{B_\gamma}|\textit{\rm{det}}A_i(y)|_pdz\simeq |{\rm det} A_i(y)|_p .p^{\gamma n}.
\end{align}
Thus, by using the H\"{o}lder inequality and (\ref{espower}), it is clear to see that
\begin{align}\label{biAiB-CMO}
&|b_{i,A_i(y)B_\gamma}-b_{i,B_{\gamma+k_{A_i}}}|\leq \dfrac{1}{|A_i(y)B_\gamma|}\int_{A_i(y)B_\gamma}|b_i(x)-b_{i,B_{\gamma+k_{A_i}}}|dx\nonumber
\\
&\leq \dfrac{1}{|A_i(y)B_\gamma|}\Big(\int_{B_{\gamma+k_{A_i}}}|b_i(x)-b_{i,B_{\gamma+k_{A_i}}}|^{r_i}\omega_i(x)dx\Big)^{\frac{1}{r_i}}\Big(\int_{B_{\gamma+k_{A_i}}}\omega_i^{\frac{-r_i'}{r_i}}dx\Big)^{\frac{1}{r_i'}}\nonumber
\\
&\leq \dfrac{\omega_i(B_{\gamma+k_{A_i}})^{\frac{1}{r_i}}.\Big(\int_{B_{\gamma+k_{A_i}}}|x|_p^{\frac{-\alpha_i r_i'}{r_i}}dx\Big)^{\frac{1}{r_i'}}}{|A_i(y)B_\gamma|}\|b_i\|_{{{{CMO}}^{r_i}_{\omega_i}(\mathbb Q^n_p)}}\nonumber
\\
&\lesssim \dfrac{p^{\frac{(\gamma +k_{A_i})(n+\alpha_i)}{r_i}}.p^{(\gamma+k_{A_i})(\frac{-\alpha_i}{r_i}+\frac{n}{r_i'})}}{|{\rm det}A_i(y)|_p.p^{\gamma n}}\|b_i\|_{{{{CMO}}^{r_i}_{\omega_i}(\mathbb Q^n_p)}}= \dfrac{\|A_i(y)\|_p^n}{|{\rm det}A_i(y)|_p}\|b_i\|_{{{{CMO}}^{r_i}_{\omega_i}(\mathbb Q^n_p)}}.
\end{align}
By using the formula for change of variables again, one has
\begin{align}\label{biAi(y)x}
&\Big(\int\limits_{B_\gamma}{|b_i(A_i(y)x)-b_{i,B_{\gamma+k_{A_i}}}|^{r_i}\omega_i(x)dx}\Big)^{\frac{1}{r_i}}\nonumber
\\
&=\Big(\int\limits_{A_i(y)B_\gamma}{|b_i(z)-b_{i,B_{\gamma+k_{A_i}}}|^{r_i}|A_i^{-1}(y)z|_p^{\alpha_i}|\textit{\rm det}A_i^{-1}(y)|_pdz}\Big)^{\frac{1}{r_i}}\nonumber
\\
&\leq \Big(\textit{\rm max}\,\lbrace \|A_i^{-1}(y)\|_p^{\alpha_i},\|A_i(y)\|_p^{-\alpha_i}\rbrace |\textit{\rm{det}}A_i^{-1}(y)|_p\,\omega_i(B_{\gamma+k_{A_i}})\Big)^{\frac{1}{r_i}}\times\nonumber
\\
&\,\,\,\,\,\,\,\,\,\,\,\,\,\,\times\Big(\dfrac{1}{\omega_i(B_{\gamma+k_{A_i}})}\int\limits_{B_{\gamma+k_{A_i}}}{|b_i(z)-b_{i, B_{\gamma+k_{A_i}}}|^{r_i}\omega_i(z)dz}\Big)^{\frac{1}{r_i}}.
\end{align}
This leads to
\begin{align}
&\Big(\int\limits_{B_\gamma}{|b_i(A_i(y)x)-b_{i,B_{\gamma +k_{A_i}}}|^{r_i}\omega_i(x)dx}\Big)^{\frac{1}{r_i}}\nonumber
\\
&\lesssim p^{\frac{\gamma(n+\alpha_i)}{r_i}}.\big(\textit{\rm max}\,\lbrace \|A_i^{-1}(y)\|_p^{\alpha_i},\|A_i(y)\|_p^{-\alpha_i}\rbrace |\textit{\rm{det}}A_i^{-1}(y)|_p\big)^{\frac{1}{r_i}}\|A_i(y)\|_p^{\frac{(n+\alpha_i)}{r_i}}.\big\|b_i\big\|_{{{CMO}}^{r_i}_{\omega_i}(\mathbb Q^n_p)}.\nonumber
\end{align}
Therefore, by (\ref{decomaI2i}) and (\ref{biAiB-CMO}), we have
\begin{align}\label{I2i}
I_{2,i}&\lesssim \Big(\dfrac{\|A_i(y)\|_p^n}{|{\rm det}A_i(y)|_p}+\big(\textit{\rm max}\,\lbrace \|A_i^{-1}(y)\|_p^{\alpha_i},\|A_i(y)\|_p^{-\alpha_i}\rbrace |\textit{\rm{det}}A_i^{-1}(y)|_p\big)^{\frac{1}{r_i}}\|A_i(y)\|_p^{\frac{(n+\alpha_i)}{r_i}}\Big)\times\nonumber
\\ 
&\,\,\,\,\,\,\,\,\,\times p^{\frac{\gamma(n+\alpha_i)}{r_i}}.\big\|b_i\big\|_{{{CMO}}^{r_i}_{\omega_i}(\mathbb Q^n_p)}.
\end{align}
Next, we consider the term $I_{3,i}$. We have
\begin{equation}\label{I3i}
I_{3,i}\leq \omega_i(B_\gamma)^{\frac{1}{r_i}}\big|b_{i, B_\gamma}-b_{i,A_i(y)B_\gamma}\big|.
\end{equation}
Fix $y\in\mathbb Q_p^n$. We set
$$
S_{k_{A_i}} = \left\{ \begin{array}{l}
\big\{j\in\mathbb Z: 1\leq j\leq k_{A_i}\big\},\,\,\,\,\,\,\,\,\,\,\,\,\,\textit{\rm if }\,k_{A_i}\geq 1,
\\
\\
\big\{j\in\mathbb Z: k_{A_i} +1\leq j\leq 0\big\},\textit{ \rm otherwise}.
\end{array} \right.
$$
As mentioned above, we have
\begin{align}\label{I3i-compose-power}
\big|b_{i,B_\gamma}-b_{i,A_i(y)B_\gamma}\big|&\leq \sum\limits_{j\in S_{k_{A_i}}}\big|b_{i,B_{\gamma+j-1}}-b_{i,B_{\gamma+j}}\big|+\big|b_{i,B_{\gamma+k_{A_i}}}-b_{i,A_i(y)B_\gamma}\big|.
\end{align}
Combining the H\"{o}lder inequality, the definition of the space $ CMO^{r_i}_{\omega_i}(\mathbb Q_p^n)$ and (\ref{espower}), one has
\begin{align}
&\big|b_{i,B_{\gamma+j-1}}-b_{i,B_{\gamma+j}}\big|\leq \dfrac{1}{|B_{\gamma+j-1}|}\int_{B_{\gamma+j}}|b_i(z)-b_{i,B_{\gamma+j}}|dz\lesssim \dfrac{1}{|B_{\gamma+j}|}\int_{B_{\gamma+j}}|b_i(z)-b_{i,B_{\gamma+j}}|dz\nonumber
\\
&\leq \dfrac{\omega_i(B_{\gamma+j})^{\frac{1}{r_i}}}{|B_{\gamma+j}|}\Big(\int_{B_{\gamma+j}}\omega_i^{\frac{-r_i'}{r_i}}dx\Big)^{\frac{1}{r_i'}}\Big(\dfrac{1}{\omega_i(B_{\gamma+j})}\int_{B_{\gamma+j}}|b_i(z)-b_{i,B_{\gamma+j}}|^{r_i}\omega_i(x)dz\Big)^{\frac{1}{r_i}}\nonumber
\\
&\lesssim\dfrac{p^{(\gamma+j)\frac{(\alpha_i+n)}{r_i}}}{p^{(\gamma+j)n}}p^{(\gamma+j)(\frac{-\alpha_i}{r_i}+\frac{n}{r_i'})}\|b_i\|_{CMO^{r_i}_{\omega_i}(\mathbb Q^n_p)}=\|b_i\|_{CMO^{r_i}_{\omega_i}(\mathbb Q^n_p)}.\nonumber
\end{align}
Thus,
\begin{align}\label{biB-biAiB}
\big|b_{i,B_\gamma}-b_{i,A_i(y)B_\gamma}\big|\lesssim |k_{A_i}|.\|b_i\|_{CMO^{r_i}_{\omega_i}(\mathbb Q^n_p)} +\big|b_{i,B_{\gamma+k_{A_i}}}-b_{i,A_i(y)B_\gamma}\big|.
\end{align}
In addition, by the H\"{o}lder inequality again and (\ref{|AiB|}), we get
\begin{align}
&\big|b_{i,B_{\gamma+k_{A_i}}}-b_{i,A_i(y)B_\gamma}\big|
\leq  \frac{1}{|A_i(y)B_\gamma|}\int\limits_{A_i(y)B_\gamma}{\big|b_i(x)-b_{i,B_{\gamma+k_{A_i}}}\big|dx}\nonumber
\\
&\leq\dfrac{\omega_i(B_{\gamma+k_{A_i}})^{\frac{1}{r_i}}}{|A_i(y)B_\gamma|}\Big(\int_{B_{\gamma+k_{A_i}}}\omega_i^{-\frac{r_i'}{r_i}}dx\Big)^{\frac{1}{r_i'}}\Big(\dfrac{1}{\omega_i(B_{\gamma+k_{A_i}})}\int\limits_{B_{\gamma+k_{A_i}}}{\big|b_i(x)-b_{i,B_{\gamma+k_{A_i}}}\big|^{r_i}\omega_i(x)dx}\Big)^{\frac{1}{r_i}}\nonumber
\\
&\lesssim \dfrac{p^{(\gamma+k_{A_i})\frac{(\alpha_i+n)}{r_i}}}{|{\rm det} A_i(y)|_p. p^{\gamma n}}.p^{(\gamma+ k_{A_i})(\frac{-\alpha_i}{r_i}+\frac{n}{r_i'})}\|b_i\big\|_{{{CMO}}^{r_i}_{\omega_i}(\mathbb Q^n_p)}=\dfrac{\|A_i(y)\|_p^n}{|{\rm det} A_i(y)|_p}\|b_i\big\|_{{{CMO}}^{r_i}_{\omega_i}(\mathbb Q^n_p)}.\nonumber
\end{align}
Consequently, by (\ref{I3i})-(\ref{biB-biAiB}), it follows that
\begin{align}
I_{3,i}&\lesssim p^{\frac{\gamma(n+\alpha_i)}{r_i}}\Big(|{\rm log}_p\|A_i(y)\|_p| + \dfrac{\|A_i(y)\|_p^n}{|{\rm det} A_i(y)|_p}\Big).\big\|b_i\big\|_{{{CMO}}^{r_i}_{\omega_i}(\mathbb Q^n_p)}.\nonumber
\end{align}
This together with (\ref{BDTphantich}), (\ref{I1i}) and (\ref{I2i}) follow us to have the proof of the inequality (\ref{BDT}). Finally, by estimating as (\ref{biAi(y)x}), we immediately have 
\begin{align}
\|f_i(A_i(y)\cdot)\|_{L^{q_i}_{\omega_i}(B_\gamma)}&\leq \Big(\textit{\rm max}\,\lbrace \|A_i^{-1}(y)\|_p^{\alpha_i},\|A_i(y)\|_p^{-\alpha_i}\rbrace |\textit{\rm{det}}A_i^{-1}(y)|_p\Big)^{\frac{1}{q_i}}\|f_i\|_{L^{q_i}_{\omega_i}(B_{\gamma+k_{A_i}})}\nonumber
\\
&=\mu_i(y).\|f_i\|_{L^{q_i}_\omega(B_{\gamma+k_{A_i}})}.\nonumber
\end{align}
In view of (\ref{HfMatrix}) and (\ref{BDT}), the proof of this lemma is ended.
\end{proof}
\begin{lemma}\label{Lemma2}
Let $1\leq q^*, r_1^*, ..., r_m^*, q_1^*, ..., q_m^*, \zeta<\infty$, $\omega\in A_{\zeta}$ with the finite critical index $r_\omega$ for the reverse H\"{o}lder condition, $\delta\in (1,r_\omega)$, $\lambda_i\in \big(\frac{-1}{q^*_{i}},0\big)$, $\zeta\leq r_i^*$ and $b_i\in {{CMO}}^{r_i^*}_{\omega}(\mathbb Q^n_p)$ for all $i=1,...,m$. Assume that the following condition holds:
\begin{align}\label{r*q*}
 \dfrac{1}{q^*}>\Big(\dfrac{1}{r_1^*}+\cdots+ \dfrac{1}{r_m^*}+ \dfrac{1}{q_1^*}+\cdots+ \dfrac{1}{q_m^*}\Big)\zeta\dfrac{r_\omega}{r_\omega-1}.
\end{align}
Then, we have
\begin{align}
\|{\mathcal H}^{p}_{\Phi,\vec A,\vec b}(\vec f)\|_{L^{q^*}_{\omega}(B_\gamma)} &\lesssim \omega(B_\gamma)^{\frac{1}{q^*}}.\mathcal B_{\vec r^*, \omega}.\Big(\int_{\mathbb Q^n_p}\dfrac{\Phi(y)}{|y|_p^n}\prod\limits_{i=1}^m\psi_i^*(y).\mu_i^*(y)\times
\nonumber
\\
&\,\,\,\,\,\,\,\,\,\,\,\,\,\,\,\times\dfrac{1}{\omega(B_{\gamma+k_{A_i}})^{\frac{1}{q_i^*}}}\|f_i\|_{L^{q_i^*}_\omega(B_{\gamma+k_{A_i}})}dy\Big), \textit{ for all }\,\gamma\in\mathbb Z.\nonumber
\end{align}
Here 
\begin{align}
\psi^*_i(y)&=1+\frac{2\|A_i(y)\|_p^{n}}{|{\rm det}A_i(y)|_p}+|{\rm det} A_i^{-1}(y)|^{\frac{\zeta}{r_i^*}}\|A_i(y)\|_p^{\frac{n\zeta}{r_i^*}}
+ |{\rm log}_p\|A_i(y)\|_p|,
\nonumber
\\
\mu^*_i(y)&=|{\rm det} A_i^{-1}(y)|_p^{\frac{\zeta}{q_i^*}}\|A_i(y)\|_p^{\frac{n\zeta}{q_i^*}}\,\textit{\rm and}\,\mathcal B_{\vec r^*,\omega}=\prod\limits_{i=1}^m\big\|b_i\big\|_{{{CMO}}^{r_i^*}_{\omega}(\mathbb Q^n_p)}.\nonumber
\end{align}
\end{lemma}
\begin{proof}
By virtue of the inequality (\ref{r*q*}), there exist $r_1, ..., r_m, q_1, ..., q_m$ such that
\begin{align}
&\dfrac{1}{q_i}>\dfrac{\zeta}{q_i^*}\dfrac{r_\omega}{r_{\omega}-1}\,\textit{\rm and }\, \dfrac{1}{r_i}>\dfrac{\zeta}{r_i^*}\dfrac{r_\omega}{r_{\omega}-1}\,\textit{\rm for all }\, i=1, ...,m,\nonumber
\\
&\,\textit{\rm and}\,\,
\dfrac{1}{q_1}+\cdots+\dfrac{1}{q_m}+\dfrac{1}{r_1}+ \cdots+\dfrac{1}{r_m}=\dfrac{1}{q^*}.\nonumber
\end{align}
As mentioned above, for any $\gamma\in\mathbb Z$, by the same argument (\ref{HfMatrix}), we also get
\begin{align}\label{HfMatrix-Ap}
\big\|\mathcal{H}_{\Phi,\vec A,\vec b}^{p}(\vec f)\big\|_{L^{q^*}_{\omega}(B_\gamma)}\lesssim \int\limits_{\mathbb Q^n_p}{\frac{\Phi(y)}{|y|^n_p}}\prod\limits_{i=1}^m\big\|b_i(\cdot)-b_i(A_i(y)\cdot)\big\|_{L^{r_i}_{\omega}(B_\gamma)}\big\|f_i(A_i(y)\cdot)\big\|_{L^{q_i}_{\omega}(B_\gamma)}dy.
\end{align}
In particular, we need to show the following result
\begin{align}\label{BDT-Ap}
&\big\|b_i(\cdot)-b_i(A_i(y)\cdot)\big\|_{L^{r_i}_{\omega}(B_\gamma)}\lesssim \omega(B_\gamma)^{\frac{1}{r_i}}.\psi_i^*(y).\big\|b_i\big\|_{{ {CMO}}^{r_i^*}_{\omega}(\mathbb Q_p^n)},
\end{align}
for all $i=1,...,m$. Indeed, we see that
\begin{align}\label{BDTphantich-Ap}
\big\|b_i(\cdot)-b_i(A_i(y)\cdot)\big\|_{L^{r_i}_{\omega}(B_\gamma)}&\leq \|b_i(\cdot)-b_{i,B_\gamma}\|_{L^{r_i}_\omega(B_\gamma)}+ \|b_i(A_i(y)\cdot)-b_{i,A_i(y)B_\gamma}\|_{L^{r_i}_\omega(B_\gamma)}\nonumber
\\
&\,\,\,\,\,\,\,+\|b_{i,B_\gamma}-b_{i,A_i(y)B_\gamma}\|_{L^{r_i}_{\omega}(B_\gamma)}:=J_{1,i}+J_{2,i}+J_{3,i}.
\end{align}
By virtue of  the inequality $r_1<r_1^*$, it is easy to show that
\begin{align}\label{J1iLemma2}
J_{1,i}\leq \omega(B_\gamma)^{\frac{1}{r_i}}\|b_i\|_{{{CMO}}^{r_i^*}_{\omega}(\mathbb Q^n_p)}.
\end{align}
Next, by estimating as (\ref{decomaI2i}) above, we get
\begin{align}\label{J2i-compose-Ap}
J_{2,i}&\leq  \omega(B_\gamma)^{\frac{1}{r_i}}|b_{i,A_i(y)B_\gamma}-b_{i,B_{\gamma+k_{A_i}}}|+\Big(\int\limits_{B_\gamma}{|b_i(A_i(y)x)-b_{i,B_{\gamma+k_{A_i}}}|^{r_i}\omega(x)dx}\Big)^{\frac{1}{r_i}}.
\end{align}
By having the inequality $\zeta\leq r_i^*$ and applying Proposition\ref{pro2.4DFan} and (\ref{|AiB|}), we infer
\begin{align}\label{I2i-biBgammakAi}
&\big|b_{i,B_{\gamma+k_{A_i}}}-b_{i,A_i(y)B_\gamma}\big|
\leq  \frac{1}{|A_i(y)B_\gamma|}\int\limits_{B_{\gamma+k_{A_i}}}{\big|b_i(x)-b_{i,B_{\gamma+k_{A_i}}}\big|dx}\nonumber
\\
&\leq \frac{|B_{\gamma+k_{A_i}}|}{|A_i(y)B_\gamma|}\frac{1}{|B_{\gamma+k_{A_i}}|}\int\limits_{B_{\gamma+k_{A_i}}}{\big|b_i(x)-b_{i,B_{\gamma+k_{A_i}}}\big|dx}\nonumber
\\
&\lesssim \frac{|B_{\gamma+k_{A_i}}|}{|A_i(y)B_\gamma|}\Big(\frac{1}{\omega(B_{\gamma+k_{A_i}})}\int\limits_{B_{\gamma+k_{A_i}}}{\big|b_i(x)-b_{i,B_{\gamma+k_{A_i}}}\big|^{\zeta}\omega(x)dx}\Big)^{\frac{1}{\zeta}}\nonumber
\\
&\leq \frac{p^{(\gamma+k_{A_i})n}}{|{\rm det}A_i(y)|_p.p^{\gamma n}}\|b_i\|_{CMO^{\zeta}_\omega(\mathbb Q^n_p)}\leq \frac{\|A_i(y)\|_p^n}{|{\rm det}A_i(y)|_p}\|b_i\|_{CMO^{r_i^*}_\omega(\mathbb Q^n_p)}.
\end{align}
By $\frac{1}{r_i}>\frac{\zeta}{r_i^*}\frac{r_\omega}{r_{\omega}-1}$, there exists $\beta_{i}\in (1, r_\omega)$ satisfying $\frac{r_i^*}{\zeta}=r_i\beta_{i}'.$ Thus, by combining the H\"{o}lder inequality and the reverse H\"{o}lder condition again, we have
\begin{align}
&\Big(\int_{B_\gamma}|b_i(A_i(y)x)-b_{i,B_{\gamma+k_{A_i}}}|^{r_i}\omega(x)dx\Big)^{\frac{1}{r_i}}\nonumber
\\
&\leq \Big(\int_{B_\gamma}{|b_i(A_i(y)x)-b_{i,B_{\gamma+k_{A_i}}}|^{\frac{r_i^*}{\zeta}}dx\Big)^{\frac{\zeta}{r_i^*}}\Big(\int_{B_\gamma}\omega(x)^{\beta_i}dx\Big)^{\frac{1}{\beta_i r_i}}}\nonumber
\\
&\lesssim |B_\gamma|^{\frac{-\zeta}{r_i^*}}\omega(B_\gamma)^{\frac{1}{r_i}}\Big(\int_{B_\gamma}{|b_i(A_i(y)x)-b_{i,B_{\gamma+k_{A_i}}}|^{\frac{r_i^*}{\zeta}}dx\Big)^{\frac{\zeta}{r_i^*}}}.\nonumber
\end{align}
According to Proposition \ref{pro2.4DFan}, we have
\begin{align}
&\Big(\int_{B_\gamma}{|b_i(A_i(y)x)-b_{i,B_{\gamma+k_{A_i}}}|^{\frac{r_i^*}{\zeta}}dx\Big)^{\frac{\zeta}{r_i^*}}}\nonumber
\\
&\leq |{\rm det} A_i^{-1}(y)|_p^{\frac{\zeta}{r_i^*}}\Big(\int_{B_{\gamma+k_{A_i}}}{|b_i(z)-b_{i,B_{\gamma+k_{A_i}}}|^{\frac{r_i^*}{\zeta}}dz\Big)^{\frac{\zeta}{r_i^*}}}\nonumber
\\
&\leq |{\rm det} A_i^{-1}(y)|_p^{\frac{\zeta}{r_i^*}}\frac{|B_{\gamma+k_{A_i}}|^{\frac{\zeta}{r_i^*}}}{\omega(B_{\gamma+k_{A_i}})^{\frac{1}{r_i^*}}}\Big(\int_{B_{\gamma+k_{A_i}}}{|b_i(z)-b_{i,B_{\gamma+k_{A_i}}}|^{r_i^*}\omega(z)dz\Big)^{\frac{1}{r_i^*}}}.\nonumber
\end{align}
In view of (\ref{espower}), one has $\frac{|B_{\gamma+k_{A_i}}|}{|B_\gamma|}\simeq \|A_i(y)\|^n_p$. From this, we give
\begin{align}\label{esbiAi_Ap-CMO}
&\Big(\int_{B_\gamma}|b_i(A_i(y)x)-b_{i,B_{\gamma+k_{A_i}}}|^{r_i}\omega(x)dx\Big)^{\frac{1}{r_i}}\nonumber
\\
&\lesssim \omega(B_\gamma)^{\frac{1}{r_i}}|{\rm det} A_i^{-1}(y)|_p^{\frac{\zeta}{r_i^*}}\|A_i(y)\|_p^{\frac{n\zeta}{r_i^*}}\times
\nonumber
\\
&\,\,\,\,\,\,\,\,\,\,\,\,\,\,\,\,\,\,\,\,\,\times\Big(\frac{1}{\omega(B_{\gamma+k_{A_i}})}\int_{B_{\gamma+k_{A_i}}}{|b_i(z)-b_{i,B_{\gamma+k_{A_i}}}|^{r_i^*}\omega(z)dz\Big)^{\frac{1}{r_i^*}}}
\\
&\lesssim \omega(B_\gamma)^{\frac{1}{r_i}}|{\rm det} A_i^{-1}(y)|_p^{\frac{\zeta}{r_i^*}}\|A_i(y)\|_p^{\frac{n\zeta}{r_i^*}}.\|b_i\|_{{{CMO}}^{r_i^*}_{\omega}(\mathbb Q^n_p)}.\nonumber
\end{align}
As a consequence, by (\ref{J2i-compose-Ap}) and (\ref{I2i-biBgammakAi}), we infer
\begin{align}\label{J2iLemma2-Ap}
J_{2,i}\lesssim \omega(B_\gamma)^{\frac{1}{r_i}}\Big(\frac{\|A_i(y)\|^n_p}{|{\rm det}A_i(y)|_p}+|{\rm det} A_i^{-1}(y)|_p^{\frac{\zeta}{r_i^*}}\|A_i(y)\|_p^{\frac{n\zeta}{r_i^*}}\Big).\|b\|_{{{CMO}}^{r_i^*}_{\omega}(\mathbb Q^n_p)}.
\end{align}
Now, we will estimate $J_{3,i}$. By a same argument as (\ref{I3i}), (\ref{I3i-compose-power}) and (\ref{I2i-biBgammakAi}), we also have
\begin{align}
J_{3,i}&\leq\omega(B_{\gamma})^{\frac{1}{r_i}}\Big(\sum\limits_{j\in S_{k_{A_i}}}\big|b_{i,B_{\gamma+j-1}}-b_{i,B_{\gamma+j}}\big|+\big|b_{i,B_{\gamma+k_{A_i}}}-b_{i,A_i(y)B_\gamma}\big|\Big)\nonumber
\\
&\lesssim\omega(B_{\gamma})^{\frac{1}{r_i}}\Big(\sum\limits_{j\in S_{k_{A_i}}}\dfrac{1}{|B_{\gamma+j}|}\int_{B_{\gamma+j}}|b_i(z)-b_{i,B_{\gamma+j}}|dz +\nonumber
\\
&\,\,\,\,\,\,\,\,\,\,\,\,\,\,\,\,\,\,\,\,\,\,\,\,\,\,\,\,\,\,\,\,\,\,\,\,\,\,\,\,\,\,\,\,\,\,\,\,\,\,\,\,\,\,\,\,\,\,\,\,+\frac{1}{|A_i(y)B_\gamma|}\int\limits_{B_{\gamma+k_{A_i}}}{\big|b_i(x)-b_{i,B_{\gamma+k_{A_i}}}\big|dx}\Big)\nonumber
\\
&\leq\omega(B_{\gamma})^{\frac{1}{r_i}}\Big(\sum\limits_{j\in S_{k_{A_i}}}\|b_i\|_{CMO_\omega^{r_i^*}(\mathbb Q^n_p)} +\frac{\|A_i(y)\|_p^n}{|{\rm det}A_i(y)|_p}\|b_i\|_{CMO^{r_i^*}_\omega(\mathbb Q^n_p)}\Big)\nonumber
\\
&\leq\omega(B_{\gamma})^{\frac{1}{r_i}}\Big(|{\rm log}_p\|A_i(y)\|_p|+\frac{\|A_i(y)\|_p^n}{|{\rm det}A_i(y)|_p}\Big)\|b_i\|_{CMO^{r_i^*}_\omega(\mathbb Q^n_p)}.\nonumber
\end{align}
This together with (\ref{J1iLemma2}) and (\ref{J2iLemma2-Ap}) yields that the inequality (\ref{BDT-Ap}) is finished.
\\
In other words, by estimating as (\ref{esbiAi_Ap-CMO}) above, we also get
\begin{align}
\|f_i(A_i(y)\cdot)\|_{L^{q_i}_\omega(B_\gamma)}&\lesssim \omega(B_\gamma)^{\frac{1}{q_i}}|{\rm det} A_i^{-1}(y)|_p^{\frac{\zeta}{q_i^*}}\|A_i(y)\|_p^{\frac{n\zeta}{q_i^*}}\omega(B_{\gamma+k_{A_i}})^{\frac{-1}{q_i^*}}\|f_i\|_{L^{q_i^*}_\omega(B_{\gamma+k_{A_i}})}.\nonumber
\\
&=\omega(B_\gamma)^{\frac{1}{q_i}}.\mu_i^*(y).\omega(B_{\gamma+k_{A_i}})^{\frac{-1}{q_i^*}}\|f_i\|_{L^{q_i^*}_\omega(B_{\gamma+k_{A_i}})}.\nonumber
\end{align}
Hence, by (\ref{HfMatrix-Ap}) and (\ref{BDT-Ap}), we conclude that the proof of this lemma is finished.
\end{proof}
\begin{theorem}
Let the assumptions of Lemma \ref{LemmaCMO1Matrix} hold and
$$
\mathcal C_5=\int_{\mathbb Q^n_p}\dfrac{\Phi(y)}{|y|_p^n}\prod\limits_{i=1}^m\psi_i(y).\mu_i(y)dy<\infty.
$$
Then, for any $\gamma\in B_{\gamma}$, we have ${\mathcal{H}}^p_{\Phi,\vec A,\vec b}$ is bounded from $L^{q_1}_{\omega_1}(\mathbb Q_p^n)\times \cdots\times L^{q_m}_{\omega_m}(\mathbb Q^n_p)$ to $L^q_{\omega}(B_\gamma)$.
\end{theorem}
\begin{proof}
For any $\gamma\in\mathbb Z$, by using Lemma \ref{LemmaCMO1Matrix}, we infer
\begin{align}
&\|{\mathcal{H}}^p_{\Phi,\vec A,\vec b}(\vec f)\|_{L^q_\omega(B_\gamma)}\lesssim p^{\sum\limits_{i=1}^m \frac{\gamma(n+\alpha_i)}{r_i}}\mathcal B_{\vec r,\vec\omega}.\int_{\mathbb Q^n_p}\dfrac{\Phi(y)}{|y|^n_p}\prod\limits_{i=1}^m\psi_i(y).\mu_i(y).\|f_i\|_{L^{q_i}_{\omega_i}(B_{\gamma+k_{A_i}})}dy.\nonumber
\end{align}
Thus, we have
$
\|{\mathcal{H}}^p_{\Phi,\vec A,\vec b}(\vec f)\|_{L^q_\omega(B_\gamma)}\lesssim \mathcal C_5.\mathcal B_{\vec r,\vec\omega}.\prod\limits_{i=1}^m\|f_i\|_{L^{q_i}_{\omega_i}(\mathbb Q^n_p)},$
which shows that the proof of this theorem is completed.
\end{proof}
\begin{theorem}
Let the assumptions of Lemma \ref{Lemma2} hold. Suppose $\omega(B_\gamma)\lesssim 1$, for all $\gamma\in\mathbb Z$, and 
\begin{align}
\mathcal C_6&=\int_{\mathbb Q^n_p}\dfrac{\Phi(y)}{|y|_p^n}\prod\limits_{i=1}^m\psi_i^*(y).\mu_i^*(y)\Big(\chi_{\{\|A_i(y)\|_p\leq 1\}}(y)\|A_i(y)\|_p^{\frac{-n\zeta}{q_i^*}}+\nonumber
\\
&\,\,\,\,\,\,\,\,\,\,\,\,\,\,\,\,\,\,+\chi_{\{\|A_i(y)\|_p>1\}(y)}\|A_i(y)\|_p^{\frac{-n(\delta-1)}{q_i^*\delta}}\Big)dy<\infty.\nonumber
\end{align}
Then, we have ${\mathcal{H}}^p_{\Phi,\vec A,\vec b}$ is bounded from $L^{q_1^*}_{\omega}(\mathbb Q_p^n)\times \cdots\times L^{q_m^*}_{\omega}(\mathbb Q^n_p)$ to $L^{q^*}_\omega(\mathbb Q^n_p)$.
\end{theorem}
\begin{proof}
In view of Lemma \ref{Lemma2}, for any $R\in\mathbb Z$, it is clear to see that
\begin{align}
&\|{\mathcal H}^{p}_{\Phi,\vec A,\vec b}(\vec f)\|_{L^{q^*}_{\omega}(B_R)}\nonumber
\\
&\lesssim \omega(B_R)^{\frac{1}{q^*}}.\mathcal B_{\vec r^*,\omega}.\Big(\int_{\mathbb Q^n_p}\dfrac{\Phi(y)}{|y|_p^n}\prod\limits_{i=1}^m\psi_i^*(y).\mu_i^*(y)\dfrac{1}{\omega(B_{R+k_{A_i}})^{\frac{1}{q_i^*}}}\|f_i\|_{L^{q_i^*}_\omega(B_{R+k_{A_i}})}dy\Big).\nonumber
\end{align}
Next, by having inequality $\frac{1}{q^*}> \frac{1}{q_1^*}+...+\frac{1}{q_m^*}$ and assuming $\omega(B_R)\lesssim 1$ for any $R\in\mathbb Z$, we have
$
\omega(B_R)^{\frac{1}{q^*}}\leq \prod\limits_{i=1}^m \omega(B_R)^{\frac{1}{q_i^*}}.
$ Thus,
\begin{align}
\|{\mathcal H}^{p}_{\Phi,\vec A,\vec b}(\vec f)\|_{L^{q^*}_{\omega}(B_R)}&\lesssim \mathcal B_{\vec r^*,\omega}.\Big(\int_{\mathbb Q^n_p}\dfrac{\Phi(y)}{|y|_p^n}\prod\limits_{i=1}^m\psi_i^*(y).\mu_i^*(y)\Big(\dfrac{\omega(B_R)}{\omega(B_{R+k_{A_i}})}\Big)^{\frac{1}{q_i^*}}dy\Big)\prod\limits_{i=1}^m\|f_i\|_{L^{q_i^*}_\omega(\mathbb Q^n_p)}.\nonumber
\\
&\lesssim \mathcal C_6.\mathcal B_{\vec r^*,\omega}.\prod\limits_{i=1}^m\|f_i\|_{L^{q_i^*}_\omega(\mathbb Q^n_p)}.\nonumber
\end{align}
Consequence, by letting $R\to +\infty$ and applying the monotone convergence theorem of Lebesgue, we also have
\begin{align}
\|{\mathcal H}^{p}_{\Phi,\vec A,\vec b}(\vec f)\|_{L^{q^*}_{\omega}(\mathbb Q^n_p)}\lesssim \mathcal C_6.\mathcal B_{\vec r^*,\omega}.\prod\limits_{i=1}^m\|f_i\|_{L^{q_i^*}_\omega(\mathbb Q^n_p)}.\nonumber
\end{align}
Therefore, the theorem is completely proved.
\end{proof}
\begin{theorem}
Let $1<\zeta<\infty$, $1\leq q^*,q_i,  r_i^*<\infty$, $-n<\alpha_i<n(\zeta-1)$, $\omega(x)=|x|_p^{\alpha}$, $\omega_i(x)=|x|_p^{\alpha_i}$, for all $i=1, ..., m$, such that
\begin{align}\label{zetaq^*}
\dfrac{\alpha_1}{q_1}+\cdots+\dfrac{\alpha_m}{q_m}=\dfrac{\zeta\alpha}{q^*},\,\, \dfrac{1}{q_1}+\cdots+\dfrac{1}{q_m}=\frac{\zeta}{q^*}, 
\end{align}
\begin{align}\label{qr-maximal}
\dfrac{1}{q_1}+\cdots+\dfrac{1}{q_m}+\dfrac{1}{r_1^*}+\cdots+\dfrac{1}{r_m^*}=1.
\end{align}
If $b_i\in CMO^{r_i^*}(\mathbb Q^n_p)$, for all $i=1,...,m,$ and 
\begin{align}
\mathcal C_7=\int_{\mathbb Q^n_p}\dfrac{\Phi(y)}{|y|_p^n}\prod\limits_{i=1}^m\Gamma_i(y)\|A_i(y)\|_p^{-\frac{(\zeta+n)}{\zeta q_i}}dy<\infty,\nonumber
\end{align}
where 
\begin{align}
\Gamma_i(y)&=\Big(1+|{\rm log}_p\|A_i(y)\|_p|+\dfrac{2\|A_i(y)\|_p^n}{|{\rm det}A_i(y)|_p}+\|A_i(y)\|_p^{\frac{n}{r_i^*}}|{\rm det}A_i^{-1}(y)|_p^{\frac{1}{r_i^*}}\Big)\times
\nonumber
\\
&\,\,\,\,\,\,\,\,\,\times |{\rm det} A_i^{-1}(y)|_p^{\frac{1}{q_i}}\|A_i(y)\|_p^{\frac{n}{q_i}},
\end{align}
then we have
\begin{align}
\|\mathcal{M}^{mod}(\mathcal H^p_{\Phi,\vec A,\vec b})(\vec f)\|_{L^{q^*}_\omega(\mathbb Q^n_p)}\lesssim \mathcal C_7.\Big(\prod\limits_{i=1}^m\|b_i\|_{CMO^{r_i^*}(\mathbb Q^n_p)}\Big).\prod\limits_{i=1}^m\|f_i\|_{L^{\zeta q_i}_{\omega_i}(\mathbb Q_p^n)}.\nonumber
\end{align}
\end{theorem}
\begin{proof}
For the sake of simplicity, we denote $\mathcal B_{\vec r^*}= \prod\limits_{i=1}^m\|b_i\|_{CMO^{r_i^*}(\mathbb Q^n_p)}$.
Now, let $x\in\mathbb Q_p^n$ and fix a ball $B_\gamma$ such that $x\in B_{\gamma}$. In view of (\ref{qr-maximal}), by using the H\"{o}lder inequality, we have
\begin{align}
&\dfrac{1}{|B_{\gamma}|}\int_{B_\gamma}|\mathcal H^p_{\Phi,\vec A,\vec b}(\vec f)(z)|dz\nonumber
\\
&\leq \dfrac{1}{|B_{\gamma}|}\int_{\mathbb Q^n_p}\dfrac{\Phi(y)}{|y|_p^n}\prod\limits_{i=1}^m\|f_i(A_i(y)\cdot)\|_{L^{q_i}(B_\gamma)}\prod\limits_{i=1}^m\|b_i(\cdot)-b_i(A_i(y)\cdot)\|_{L^{r_i^*}(B_{\gamma})}dy.\nonumber
\end{align}
For $i=1,...,m$, by estimating as (\ref{BDT}) above, we also get
\begin{align}
&\big\|b_i(\cdot)-b_i(A_i(y)\cdot)\big\|_{L^{r_i^*}(B_\gamma)}\nonumber
\\
&\lesssim |B_\gamma|^{\frac{1}{r_i^*}}\Big(1+|{\rm log}_p\|A_i(y)\|_p|+\dfrac{2\|A_i(y)\|_p^n}{|{\rm det}A_i(y)|_p}+\|A_i(y)\|_p^{\frac{n}{r_i^*}}|{\rm det}A_i^{-1}(y)|_p^{\frac{1}{r_i^*}}\Big).\big\|b_i\big\|_{{ {CMO}}^{r_i^*}(\mathbb Q_p^n)}.\nonumber
\end{align}
By $x\in B_{\gamma}$, we imply that $\|A_i(y)\|^{-1}_p.x\in B_{\gamma+k_{A_i}}$
Thus, by definition of the Hardy-Littlewood maximal  operator, one has
\begin{align}
\|f_i(A_i(y)\cdot)\|_{L^{q_i}(\mathbb Q^n_p)}&=|{\rm det}A_i^{-1}(y)|_p^{\frac{1}{q_i}}\Big(\int_{A_i(y)B_\gamma}|f_i(t)|^{q_i}dt\Big)^{\frac{1}{q_i}}\leq |{\rm det}A_i^{-1}(y)|_p^{\frac{1}{q_i}}\Big(\int_{B_{\gamma+k_{A_i}}}|f_i(t)|^{q_i}dt\Big)^{\frac{1}{q_i}}\nonumber
\\
&=\dfrac{|B_{\gamma+k_{A_i}}|^{\frac{1}{q_i}}}{|B_\gamma|^{\frac{1}{q_i}}}|B_\gamma|^{\frac{1}{q_i}}|{\rm det}A_i^{-1}(y)|_p^{\frac{1}{q_i}}\Big(\dfrac{1}{|B_{\gamma+k_{A_i}}|}\int_{B_{\gamma+k_{A_i}}}|f_i(t)|^{q_i}dt\Big)^{\frac{1}{q_i}}\nonumber
\\
&\lesssim |B_\gamma|^{\frac{1}{q_i}}|{\rm det}A_i^{-1}(y)|_p^{\frac{1}{q_i}}\|A_i(y)\|_p^{\frac{n}{q_i}}\Big(\mathcal{M}(|f_i|^{q_i})(\|A_i(y)\|_p^{-1}.x)\Big)^{\frac{1}{q_i}}.\nonumber
\end{align}
As mentioned above, we give 
\begin{align}
\dfrac{1}{|B_{\gamma}|}\int_{B_\gamma}|\mathcal H^p_{\Phi,\vec A,\vec b}(\vec f)(z)|dz &\lesssim \mathcal B_{\vec r^*}.\int_{\mathbb Q^n_p}\dfrac{\Phi(y)}{|y|_p^n}\prod\limits_{i=1}^m\Gamma_i(y)\Big(\mathcal{M}(|f_i|^{q_i})(\|A_i(y)\|_p^{-1}.x)\Big)^{\frac{1}{q_i}}dy.\nonumber
\end{align}
Thus, we infer that
\begin{align}
\mathcal{M}^{mod}(\mathcal H^p_{\Phi,\vec A,\vec b}(\vec f))(x)&\lesssim \mathcal B_{\vec r^*}.\int_{\mathbb Q^n_p}\dfrac{\Phi(y)}{|y|_p^n}\prod\limits_{i=1}^m\Gamma_i(y)\Big(\mathcal{M}(|f_i|^{q_i})(\|A_i(y)\|_p^{-1}.x)\Big)^{\frac{1}{q_i}}dy.\nonumber
\end{align}
Thus, by  using the assumption (\ref{zetaq^*}) and the Minkowski inequality and the H\"{o}lder inequality, we obtain
\begin{align}\label{MmodHp}
&\|\mathcal{M}^{mod}(\mathcal H^p_{\Phi,\vec A,\vec b}(\vec f))\|_{L^{q^*}_\omega(\mathbb Q^n_p)}\nonumber
\\
&\leq\mathcal B_{\vec r^*}.\int_{\mathbb Q^n_p}\dfrac{\Phi(y)}{|y|_p^n}\prod\limits_{i=1}^m\Gamma_i(y)\Big(\int_{\mathbb Q^n_p}\prod\limits_{i=1}^m\Big(\mathcal{M}(|f_i|^{q_i})(\|A_i(y)\|_p^{-1}.x)\Big)^{\frac{q^*}{q_i}}\omega(x)dx\Big)^{\frac{1}{q^*}}dy
\nonumber
\\
&=\mathcal B_{\vec r^*}.\int_{\mathbb Q^n_p}\dfrac{\Phi(y)}{|y|_p^n}\prod\limits_{i=1}^m\Gamma_i(y)\Big(\int_{\mathbb Q^n_p}\prod\limits_{i=1}^m\Big(\mathcal{M}(|f_i|^{q_i})(\|A_i(y)\|_p^{-1}.x)\omega_i^{\frac{1}{\zeta}}\Big)^{\frac{q^*}{q_i}}dx\Big)^{\frac{1}{q^*}}dy
\nonumber
\\
&\leq \mathcal B_{\vec r^*}.\int_{\mathbb Q^n_p}\dfrac{\Phi(y)}{|y|_p^n}\prod\limits_{i=1}^m\Gamma_i(y)\prod\limits_{i=1}^m\Big(\int_{\mathbb Q^n_p}\mathcal{M}(|f_i|^{q_i})^{\zeta}(\|A_i(y)\|_p^{-1}.x)\omega_idx\Big)^{\frac{1}{\zeta q_i}}dy.
\end{align}
For $i=1,...,m$, by Proposition \ref{pro_power}, we have $\omega_i\in A_{\zeta}$. From this, by virtue of the boundedness of the Hardy-Littlewood maximal operator on the Lebesgue spaces with the Muckenhoupt weights, we have that
\begin{align}
&\Big(\int_{\mathbb Q^n_p}M(|f_i|^{q_i})^{\zeta}(\|A_i(y)\|_p^{-1}.x)\omega_idx\Big)^{\frac{1}{\zeta q_i}}=\Big(\int_{\mathbb Q^n_p}M(|f_i|^{q_i})^{\zeta}(z).|\,\|A_i(y)\|_pz|_p^{\alpha_i}.|\,\|A_i(y)\|_p^n|_pdz\Big)^{\frac{1}{\zeta q_i}}\nonumber
\\
&=\|A_i(y)\|_p^{\frac{-(\alpha_i+n)}{\zeta q_i}}\Big(\int_{\mathbb Q^n_p}M(|f_i|^{q_i})^{\zeta}(z)\omega_i(z)dz\Big)^{\frac{1}{\zeta q_i}}\lesssim \|A_i(y)\|_p^{\frac{-(\alpha_i+n)}{\zeta q_i}}\|f_i\|_{L^{\zeta q_i}_{\omega_i}(\mathbb Q^n_p)}.\nonumber
\end{align}
This together with (\ref{MmodHp}) yields that the proof of this theorem is completed.
\end{proof}
In what follows, we set
$$
\mathcal C_8=\int_{\mathbb Q^n_p}\dfrac{\Phi(y)}{|y|_p^n}\prod\limits_{i=1}^m \|A_i^{-1}(y)\|_p^{-(\alpha_i+n)\lambda_i}|{\rm log_p}\|A_i(y)\|_p|dy.
$$
\begin{theorem}
Suppose the hypothesis in Lemma \ref{LemmaCMO1Matrix} holds. Let $\lambda_i\in (\frac{-1}{q_i},0)$ for all $i=1,...,m$,  and conditions (\ref{DK1}) and (\ref{lambdaMorrey}) be true. Assume that
\begin{align}\label{|A|<1}
 {\rm supp}(\Phi)\subset \mathop\cap\limits_{i=1}^m \{y\in\mathbb Q^n_p:\|A_i(y)\|_p<1\}.
\end{align}
(i)\,If $\mathcal C_8<\infty$, then ${\mathcal{H}}^p_{\Phi,\vec A,\vec b}$ is bounded from ${\mathop{B}\limits^.}^{q_1,\lambda_1}_{\omega_1}(\mathbb Q_p^n)\times \cdots\times {\mathop{B}\limits^.}^{q_m, \lambda_m}_{\omega_m}(\mathbb Q^n_p)$ to ${\mathop{B}\limits^.}^{q,\lambda}_\omega(\mathbb Q_p^n)$.
\\
(ii) If\, ${\mathcal{H}}^p_{\Phi,\vec A,\vec b}$ is bounded from ${\mathop{B}\limits^.}^{q_1,\lambda_1}_{\omega_1}(\mathbb Q_p^n)\times \cdots\times {\mathop{B}\limits^.}^{q_m, \lambda_m}_{\omega_m}(\mathbb Q^n_p)$ to ${\mathop{B}\limits^.}^{q,\lambda}_\omega(\mathbb Q_p^n)$ for all $\vec b=(b_1, ..., b_m)\in CMO_{\omega_1}^{r_1}(\mathbb Q^n_p)\times\cdots\times
CMO_{\omega_m}^{r_m}(\mathbb Q^n_p)$, then $\mathcal C_8<\infty$.
\end{theorem}
\begin{proof}
Firstly, we prove the part $(\rm i)$ of the theorem. For any $R\in\mathbb Z$, by Lemma \ref{LemmaCMO1Matrix}, we get
\begin{align}
&\frac{1}{\omega(B_R)^{\frac{1}{q}+\lambda}}\|{\mathcal{H}}^p_{\Phi,\vec A,\vec b}(\vec f)\|_{L^q_\omega(B_R)}\lesssim p^{\sum\limits_{i=1}^m \frac{R(n+\alpha_i)}{r_i}}.\mathcal B_{\vec r,\vec\omega}.\int_{\mathbb Q^n_p}\dfrac{\Phi(y)}{|y|^n_p}\prod\limits_{i=1}^m\psi_i(y).\mu_i(y)\times\nonumber
\\
&\times\dfrac{1}{\omega_i(B_{R+k_{A_i}})^{\frac{1}{q_i}+\lambda_i}}\|f_i\|_{L^{q_i}_{\omega_i}(B_{R+k_{A_i}})}\dfrac{\prod\limits_{i=1}^m \omega_i(B_{R+k_{A_i}})^{\frac{1}{q_i}+\lambda_i}}{\omega(B_R)^{\frac{1}{q}+\lambda}}dy\nonumber
\\
&\leq \mathcal B_{\vec r,\vec\omega}.\Big(\int_{\mathbb Q^n_p}\dfrac{\Phi(y)}{|y|^n_p}\prod\limits_{i=1}^m\psi_i(y).\mu_i(y).\dfrac{p^{\sum\limits_{i=1}^m \frac{R(n+\alpha_i)}{r_i}}\prod\limits_{i=1}^m \omega_i(B_{R+k_{A_i}})^{\frac{1}{q_i}+\lambda_i}}{\omega(B_R)^{\frac{1}{q}+\lambda}}dy\Big)\prod\limits_{i=1}^m\|f_i\|_{{\mathop{B}\limits^{.}}^{q_i,\lambda_i}_{\omega_i}(\mathbb Q^n_p)}.\nonumber
\end{align}
Now, by (\ref{espower}) and (\ref{lambdaMorrey}), we calculate
\begin{align}
\dfrac{p^{\sum\limits_{i=1}^m \frac{R(n+\alpha_i)}{r_i}}\prod\limits_{i=1}^m \omega_i(B_{R+k_{A_i}})^{\frac{1}{q_i}+\lambda_i}}{\omega(B_R)^{\frac{1}{q}+\lambda}}&\simeq \dfrac{p^{\sum\limits_{i=1}^m \frac{R(n+\alpha_i)}{r_i}+\sum\limits_{i=1}^m(R+k_{A_i})(\alpha_i+n)(\frac{1}{q_i}+\lambda_i)}}{p^{R(\alpha+n)(\frac{1}{q}+\lambda)}}
\nonumber
\\
&=p^{\sum\limits_{i=1}^m k_{A_i}(\alpha_i+n)(\frac{1}{q_i}+\lambda_i)}=\prod\limits_{i=1}^m\|A_i(y)\|_p^{(\alpha_i+n)(\frac{1}{q_i}+\lambda_i)}.\nonumber
\end{align}
Hence, one has
\begin{align}\label{Morrey-CMO-Hp}
\|{\mathcal{H}}^p_{\Phi,\vec A,\vec b}(\vec f)\|_{{\mathop{B}\limits^{.}}^{q,\lambda}_{\omega}(\mathbb Q^n_p)}&\lesssim \mathcal B_{\vec r,\vec\omega}.\Big(\int_{\mathbb Q^n_p}\dfrac{\Phi(y)}{|y|^n_p}\prod\limits_{i=1}^m\psi_i(y)\mu_i(y)\|A_i(y)\|_p^{(\alpha_i+n)(\frac{1}{q_i}+\lambda_i)}dy\Big)\times
\nonumber
\\
&\,\,\,\,\,\times\prod\limits_{i=1}^m\|f_i\|_{{\mathop{B}\limits^{.}}^{q_i,\lambda_i}_{\omega_i}(\mathbb Q^n_p)}.
\end{align}
Note that, by the hypothesis (\ref{|A|<1}), we see that
$$
|{\rm log}_p\|A_i(y)\|_p|\geq 1,\,\textit{\rm for all}\, y\in {\rm supp}(\Phi).
$$
As mentioned above, by (\ref{|det|}) and (\ref{DK1}), we make
\begin{align}
&\psi_i(y)\mu_i(y)\|A_i(y)\|_p^{(\alpha_i+n)(\frac{1}{q_i}+\lambda_i)}\nonumber
\\
&\lesssim \big(1+\|A_i^{-1}(y)\|_p^{\frac{(\alpha_i+n)}{r_i}}\|A_i^{-1}(y)\|_p^{\frac{-(\alpha_i+n)}{r_i}}+|{\rm log}_p\|A_i(y)\|_p| +2p^{\nu_{\vec{A} }}\big)\times\nonumber
\\
&\,\,\,\,\,\,\,\times\|A_i^{-1}(y)\|_p^{\frac{(\alpha_i+n)}{q_i}}\|A_i^{-1}(y)\|_p^{-(\alpha_i+n)(\frac{1}{q_i}+\lambda_i)}
\nonumber
\\
&\lesssim |{\rm log}_p\|A_i(y)\|_p|.\|A_i(y)\|_p^{-(\alpha_i+n)\lambda_i}.\nonumber
\end{align}
As an application, by (\ref{Morrey-CMO-Hp}), we obtain that
\begin{align}
&\|{\mathcal{H}}^p_{\Phi,\vec A,\vec b}(\vec f)\|_{{\mathop{B}\limits^{.}}^{q,\lambda}_{\omega}(\mathbb Q^n_p)}\lesssim \mathcal C_8.\mathcal B_{\vec r,\vec\omega}.\prod\limits_{i=1}^m\|f_i\|_{{\mathop{B}\limits^{.}}^{q_i,\lambda_i}_{\omega_i}(\mathbb Q^n_p)}.\nonumber
\end{align}

To give the proof for the part $(\rm ii)$ of the theorem, for $i=1,...,m$, let us choose
$
b_i(x)={\rm log}_p|x|_p$  for all $ x\in\mathbb Q^n_p\setminus\{0\}$, and ${f}_i(x)=|x|_p^{(n+\alpha_i)\lambda_i}$  for all $x\in\mathbb Q^n_p.
$
Now, we need to prove that
\begin{align}\label{choose-bi}
\|b_i\|_{CMO_{\omega_i}^{r_i}(\mathbb Q^n_p)}<\infty,\,\textit{\rm for all} \,\,i=1,...,m.
\end{align}
In fact, for any $R\in\mathbb Z$, we see that
\begin{align}
b_{i,R}&=\frac{1}{|B_R|}\int_{B_R}{\rm log}_p|x|_pdx=p^{-Rn}\sum\limits_{\gamma\leq R}\int_{S_\gamma}\gamma dx= p^{-Rn}\sum\limits_{\gamma\leq R}\gamma p^{\gamma n}(1-p^{-n})\nonumber
\\
&=p^{-Rn}(1-p^{-n})\frac{p^n.p^{Rn}(Rp^n-R-1)}{(p^{n}-1)^2}=\frac{Rp^n-R-1}{p^n-1}=R-\frac{1}{p^n-1}.\nonumber
\end{align}
Thus, we get
\begin{align}
&\frac{1}{\omega_i(B_R)}\int_{B_R}|b_i(x)-b_{i,B_R}|^{r_i}\omega_idx=p^{-R(\alpha_i+n)}\sum\limits_{\gamma\leq R}\int_{S_\gamma}\left|\gamma-(R-\frac{1}{p^n-1})\right|^{r_i}p^{\gamma\alpha_i}dx\nonumber
\\
&=p^{-R(\alpha_i+n)}\sum\limits_{\gamma\leq R} \left|\gamma-(R-\frac{1}{p^n-1})\right|^{r_i}p^{\gamma(\alpha_i+n)}(1-p^{-n})\nonumber
\\
&\lesssim p^{-R(\alpha_i+n)}\sum\limits_{\gamma\leq R} \left|\gamma-(R-\frac{1}{p^n-1})\right|^{r_i}p^{\gamma(\alpha_i+n)}= p^{-R(\alpha_i+n)}\sum\limits_{\ell\leq 0} \left|\ell+\frac{1}{p^n-1})\right|^{r_i}p^{(R+\ell)(\alpha_i+n)}\nonumber
\\
&\leq \sum\limits_{\ell\leq 0} \big(|\ell|^{r_i}+\frac{1}{(p^n-1)^{r_i}}\big)p^{\ell(\alpha_i+n)}<\infty,\nonumber
\end{align}
uniformly for $R\in\mathbb Z$. As an application, it immediately follows that the inequality (\ref{choose-bi}) holds. By choosing $b_i$ and $f_i$, we have
\begin{align}
\mathcal H^p_{\Phi,\vec A,\vec b}(\vec f)(x)=\int_{\mathbb Q^n_p}\dfrac{\Phi(y)}{|y|_p^n}\prod\limits_{i=1}^m |A_i(y)x|_p^{(\alpha_i+n)\lambda_i}\big({\rm log}_p\frac{|x|_p}{|A_i(y)x|_p}\big)dy.\nonumber
\end{align}
By the hypothesis (\ref{|A|<1}), we have $\|A_i(y)\|_p<1$, for all $y\in {\rm supp}(\Phi)$. Thus,
$$
|A_i(y)x|_p\leq \|A_i(y)\|_p.|x|_p<|x|_p.
$$
This gives that
$$
0< |{\rm log}_p \|A_i(y)\|_p|={\rm log}_p\frac{1}{|A_i(y)|_p}\leq {\rm log}_p\frac{|x|_p}{|A_i(y)x|_p}.
$$
Consequently, by (\ref{DK1}) and (\ref{LemmaCMO1Matrix}), we lead to
\begin{align}
\mathcal H^p_{\Phi,\vec A,\vec b}(\vec f)(x)&\gtrsim \Big(\int_{\mathbb Q^n_p}\dfrac{\Phi(y)}{|y|_p^n}\prod\limits_{i=1}^m \|A_i(y)\|_p^{-(\alpha_i+n)\lambda_i}|{\rm log}_p \|A_i(y)\|_p|dy\Big)|x|_p^{(\alpha+n)\lambda}\nonumber
\\
&=\mathcal C_8.|x|_p^{(\alpha+n)\lambda}.\nonumber
\end{align}
From this, by (\ref{cal-fi-Morrey}) above, we infer
\begin{align}
\|\mathcal H^p_{\Phi,\vec A,\vec b}(\vec f)\|_{{\mathop{B}\limits^.}^{q,\lambda}_{\omega}(\mathbb Q^n_p)}\gtrsim \mathcal C_8.\|\,|\cdot|_p^{(\alpha+n)\lambda}\|_{{\mathop{B}\limits^.}^{q,\lambda}_{\omega}(\mathbb Q^n_p)}\gtrsim \mathcal C_8.\prod\limits_{i=1}^m \|f_i\|_{{\mathop{B}\limits^.}^{q_i,\lambda_i}_{\omega_i}(\mathbb Q^n_p)}.\nonumber
\end{align}
Therefore, since  ${\mathcal{H}}^p_{\Phi,\vec A,\vec b}$ is bounded from ${\mathop{B}\limits^.}^{q_1,\lambda_1}_{\omega_1}(\mathbb Q_p^n)\times \cdots\times {\mathop{B}\limits^.}^{q_m, \lambda_m}_{\omega_m}(\mathbb Q^n_p)$ to ${\mathop{B}\limits^.}^{q,\lambda}_\omega(\mathbb Q_p^n)$, it implies that $\mathcal C_8<\infty$. This leads to that the theorem is completely proved.
\end{proof}
Now, we consider $A_i(y)= s_i(y).I_n$, for $i=1,...,m$. By the similar arguments, we then  obtain the following useful result.
\begin{corollary}
Suppose the hypothesis in Lemma \ref{LemmaCMO1Matrix} and (\ref{lambdaMorrey}) hold. Denote 
$$
\mathcal C_9:= \int_{\mathbb Q^n_p}\dfrac{\Phi(y)}{|y|_p^n}\prod\limits_{i=1}^m |s_i(y)|_p^{(\alpha_i+n)\lambda_i}|{\rm log_p}|s_i(y)|_p|dy.
$$
Then, the following statements are equivalent:
\\
(i)\, ${\mathcal{H}}^p_{\Phi,\vec A,\vec b}$ is bounded from ${\mathop{B}\limits^.}^{q_1,\lambda_1}_{\omega_1}(\mathbb Q_p^n)\times \cdots\times {\mathop{B}\limits^.}^{q_m, \lambda_m}_{\omega_m}(\mathbb Q^n_p)$ to ${\mathop{B}\limits^.}^{q,\lambda}_\omega(\mathbb Q_p^n)$, for any $\vec b=(b_1, ..., b_m)\in CMO_{\omega_1}^{r_1}(\mathbb Q^n_p)\times\cdots\times
CMO_{\omega_m}^{r_m}(\mathbb Q^n_p)$.
\\
(ii) $\mathcal C_9<\infty$.
\end{corollary}
\begin{theorem}
Suppose the hypothesis in Lemma \ref{Lemma2} holds. Let $\lambda_i\in (\frac{-1}{q_i^*},0)$ for all $i=1,...,m$ and condition (\ref{lambdaMorreyAp}) in Theorem \ref{theo-HfMorreyq*} hold. Then, if
\begin{align}
\mathcal C_{10}&=\int_{\mathbb Q_p^n}\dfrac{\Phi(y)}{|y|_p^n}\prod\limits_{i=1}^m \psi_i^*(y).\mu_i^*(y)\times\nonumber
\\
&\times\Big(\chi_{\{\|A_i(y)\|_p\leq 1\}}(y)\|A_i(y)\|_p^{n\zeta\lambda_i} +\chi_{\{\|A_i(y)\|_p>1\}(y)}\|A_i(y)\|_p^{\frac{n\lambda_i(\delta-1)}{\delta}}\Big)dy<\infty,\nonumber
\end{align}
 we have ${\mathcal H}^p_{\Phi,\vec A,\vec b}$ is bounded from ${\mathop{B}\limits^.}^{q_1^*,\lambda_1}_{\omega}(\mathbb Q_p^n)\times \cdots\times {\mathop{B}\limits^.}^{q_m^*, \lambda_m}_{\omega}(\mathbb Q^n_p)$ to ${\mathop{B}\limits^.}^{q^*,\lambda}_\omega(\mathbb Q_p^n).$
\end{theorem}
\begin{proof}
For any $R\in\mathbb Z$, by Lemma \ref{Lemma2} and (\ref{lambdaMorreyAp}), we infer
\begin{align}
&\dfrac{1}{\omega(B_R)^{\frac{1}{q^*}+\lambda}}\|{\mathcal H}^{p}_{\Phi,\vec A,\vec b}(\vec f)\|_{L^{q^*}_{\omega}(B_R)}\nonumber
\\
&\lesssim \mathcal B_{\vec r^*, \omega}.\Big(\int_{\mathbb Q^n_p}\dfrac{\Phi(y)}{|y|_p^n}\prod\limits_{i=1}^m\psi_i^*(y).\mu_i^*(y)\dfrac{1}{\omega(B_{R+k_{A_i}})^{\frac{1}{q_i^*}+\lambda_i}}\|f_i\|_{L^{q_i^*}_\omega(B_{R+k_{A_i}})}\Big(\dfrac{\omega(B_{R+k_{A_i}})}{\omega(B_R)}\Big)^{\lambda_i}dy\Big)\nonumber
\\
&\lesssim \mathcal B_{\vec r^*, \omega}.\Big(\int_{\mathbb Q^n_p}\dfrac{\Phi(y)}{|y|_p^n}\prod\limits_{i=1}^m\psi_i^*(y).\mu_i^*(y)\Big(\dfrac{\omega(B_{R+k_{A_i}})}{\omega(B_R)}\Big)^{\lambda_i}dy\Big)\prod\limits_{i=1}^m\|f_i\|_{{\mathop{B}\limits^{.}}^{q^*_i,\lambda_i}_\omega(\mathbb Q^n_p)}.\nonumber
\end{align}
From this, by (\ref{theo3.4-omega}), we have
$
\|\mathcal H^p_{\Phi,\vec A,\vec b}(\vec f)\|_{{\mathop{B}\limits^.}^{q^*,\lambda}_{\omega}(\mathbb Q_p^n)}\lesssim \mathcal C_{10}.\mathcal B_{\vec r^*, \omega}. \prod\limits_{i=1}^m \|f_i\|_{{\mathop{B}\limits^.}^{q_i^*,\lambda_i}_{\omega}(\mathbb Q_p^n)},
$
which implies that the proof of this theorm is finished.

\end{proof}
\bibliographystyle{amsplain}

\end{document}